\documentclass{amsart}
\usepackage[mathscr]{eucal}% So-called Euler fonts, allows \mathscr
\usepackage{amssymb}% allows special arrows like >-> e.g.
\usepackage{pifont}
\usepackage{graphicx}
\usepackage{psfrag} % Necessary to put latex into .eps text fragments
\usepackage[usenames,dvipsnames]{color}
\usepackage[normalem]{ulem}% allows \sout to cross-out text.
\usepackage{amsthm}% allows theorem* , e.g.
\usepackage{bbold}% allows \unit \mathbb{1}
\usepackage{bbm}% same as above but slightly less nice; use if \mathbb should remain as `normal'
\usepackage{enumerate}% allows easy change of label
\usepackage{array}% allows array
\usepackage{amsmath}% allows pmatrix

\usepackage{hyperref}
\hypersetup{colorlinks=true,linkcolor={Brown},citecolor={Brown},urlcolor={Brown}}% replace BrickRed by Brown?

\usepackage{cleveref}% allows references via \cref and \Cref which automatically repeat the name (Thm, Prop, etc) of the ref.

\numberwithin{equation}{section}% makes equat numb contain the section % superseded by the above
\makeatletter
\@namedef{subjclassname@2020}{%
  \textup{2020} Mathematics Subject Classification}
\makeatother

\makeindex

\usepackage[all]{xy}
\CompileMatrices

\newdir{ >}{{}*!/-10pt/\dir{>}}

\swapnumbers %pour que le numéro apparaisse devant le théorème

\newtheorem{Thm}[equation]{Theorem}
\newtheorem*{Thm*}{Theorem}

\newtheorem{Prop}[equation]{Proposition}
\newtheorem{Lem}[equation]{Lemma}
\newtheorem{Cor}[equation]{Corollary}

\theoremstyle{remark}
\newtheorem{Rem}[equation]{Remark}
\newtheorem{Def}[equation]{Definition}
\newtheorem{Ter}[equation]{Terminology}

\newtheorem{Exa}[equation]{Example}

\newtheorem{Conv}[equation]{Convention}

\theoremstyle{definition}

\newtheorem*{Convs*}{Conventions}
\newtheorem*{Ack*}{Acknowledgements}
\newtheorem*{Org*}{Organization}

\newcommand{\nc}{\newcommand}
\nc{\dmo}{\DeclareMathOperator}

\dmo{\Ab}{Ab}
\dmo{\add}{add} % the additive hull (= closure under sums and direct summands)
\dmo{\Aut}{Aut}
\dmo{\bickMack}{\biMack^{\mathsf{ic}}_\kk} % the bicategory of i.c. Mackey 2-functors
\dmo{\biMack}{\mathsf{Mack}} % the bicategory of Mackey 2-functors
\dmo{\BBouc}{\cat{B}}
\dmo{\Ch}{Ch}% ground notation for chain complexes
\dmo{\CoInd}{CoInd}
\dmo{\Der}{D}% ground notation for derived categories
\dmo{\DER}{\mathsf{DER}}
\dmo{\Db}{D^b}% ground notation for bounded derived categories
\dmo{\End}{End}
\dmo{\Endbicat}{\mathsf{End}}
\dmo{\Fun}{\mathrm{Fun}} % the 2-category of 2-functors from enriched category theory
\dmo{\Free}{\mathrm{free}}
\dmo{\Cofree}{\mathrm{cofree}}
\dmo{\forget}{\mathrm{forget}}
\dmo{\Hom}{Hom}
\dmo{\Ho}{Ho}
\dmo{\img}{Im}
\dmo{\incl}{incl}
\dmo{\Ind}{Ind}
\dmo{\inj}{in} % notation for canonical injections
\dmo{\Inj}{Inj} % injective modules/objects
\dmo{\Ker}{Ker}
\dmo{\Kadd}{K_0^{add}}
\dmo{\Kex}{K_0^{exa}}
\dmo{\Kexhigher}{K^{exa}_*}
\dmo{\Ktr}{K_0^{tri}}
\dmo{\Mackey}{Mack} % the category of ordinary Mackey functors
\dmo{\CohMackey}{CohMack} % ordinary cohomological Mackey functors
\dmo{\Map}{Map}%
\dmo{\Mod}{Mod}% sheaves of modules
\dmo{\Modbicat}{\mathsf{Mod}}
\dmo{\Comod}{Comod}
\dmo{\Nat}{Nat}
\dmo{\Qcoh}{Qcoh}% quasi-coherent sheaves over a scheme
\dmo{\coh}{coh}% coherent sheaves over a scheme
\dmo{\fgmod}{mod}
\dmo{\fgfree}{free}
\dmo{\stmod}{stmod}
\dmo{\StMod}{StMod}
\dmo{\latt}{latt}
\dmo{\Mor}{Mor}%
\dmo{\Obj}{Obj}
\dmo{\Proj}{Proj} % projective modules/objects
\dmo{\fgproj}{proj} % fg projective modules/objects
\dmo{\pr}{pr}
\dmo{\2Fun}{\mathsf{2Fun}}
\dmo{\PsFunJJ}{\PsFun_{\JJ_!}^{\JJ^\prime\textrm{\!-}\mathsf{oplax}}}
\dmo{\PsFunJop}{\PsFun_{{{\JJ}_{{}_{*}}}}}
\dmo{\PsFunJ}{\PsFun_{\JJ_!}}
\dmo{\PsFunoplax}{\PsFun^{\mathsf{oplax}}}
\dmo{\PsFun}{\mathsf{PsFun}} % the bicategory of pseudofunctors
\dmo{\PsNat}{\mathsf{PsNat}}
\dmo{\PsMon}{\mathsf{PsMon}} % 2-cat of pseudo-monoids
\dmo{\PsMod}{\mathsf{PsMod}} % 2-cat of (left) pseudo-modules
\dmo{\BrPsMon}{\mathsf{BrPsMon}}
\dmo{\SymPsMon}{\mathsf{SymPsMon}}
\dmo{\Rad}{Rad}
\dmo{\Res}{Res}
\dmo{\SH}{SH}% ground name for cat of spectra
\dmo{\Sh}{Sh}
\dmo{\Spanname}{{\sf Span}}
\dmo{\Spec}{Spec}
\dmo{\Stab}{Stab}% stable category of non-fin. gen. mod.
\dmo{\twoFun}{2\mathsf{Fun}}
\dmo{\tr}{tr}

\nc{\ababs}{{\sl ab absurdo}}
\nc{\Add}{\mathsf{Add}}
\nc{\ADD}{\mathsf{ADD}}
\nc{\ADDic}{\mathsf{ADD}^{\ic}}
\nc{\ADDer}{\mathsf{ADDer}}
\nc{\ADDick}{\mathsf{ADD}_\kk{}^{\!\!(\ic)}} % one or the other
\nc{\adhoc}{{\sl ad hoc}}
\nc{\adjto}{\rightleftarrows}
\nc{\adj}{\dashv\,}
\nc{\afortiori}{{\sl a fortiori}}
\nc{\aka}{{a.\,k.\,a.}\ }
\nc{\all}{\mathsf{all}}% all 1-cells.
\nc{\apriori}{{\sl a priori}}
\nc{\ass}{\mathrm{ass}} % associator
\nc{\bbA}{\mathbb{A}}
\nc{\bbB}{\mathbb{B}}
\nc{\bbC}{\mathbb{C}}
\nc{\bbD}{\mathbb{D}}
\nc{\bbF}{\mathbb{F}}
\nc{\bbI}{\mathbb{I}}
\nc{\bbM}{\mathbb{M}}
\nc{\bbN}{\mathbb{N}}
\nc{\bbP}{\mathbb{P}}
\nc{\bbQ}{\mathbb{Q}}
\nc{\bbR}{\mathbb{R}}
\nc{\bbZ}{\mathbb{Z}}
\nc{\bs}{\backslash}
\nc{\BurnG}{\cat{A}(G)}
\nc{\cat}[1]{\mathcal{#1}}
\nc{\Cat}{\mathsf{Cat}}
\nc{\CAT}{\mathsf{CAT}}
\nc{\cf}{{\sl cf.}\ }
\nc{\Cf}{{\sl Cf.}\ }
\nc{\colim}{\mathop{\mathrm{colim}}}
\nc{\costar}{**}% for the (-)^* embedding with \beta^{-1} on 2-cells
\nc{\co}{{\mathrm{co}}}
\nc{\DD}{\cat{D}}% a derivator
\nc{\Displ}{\displaystyle}
\nc{\diag}[1]{\overline{#1}} % (essential) diagonal-part operation
\nc{\offdiag}[1]{{#1}^\dagger} % off-diagonal-part operation
\nc{\doublequot}[3]{#1\backslash #2/#3}% double-cosets
\nc{\Ecell}{\rotatebox[origin=c]{90}{$\Downarrow$}} % treated as others too! But not to be used in-line!!! [I]
\nc{\eg}{{\sl e.g.}\ } % made similar to the previous two
\nc{\Eg}{{\sl E.g.}\ } % we needed a capital one!
\nc{\eps}{\varepsilon}
\nc{\equalby}[1]{\overset{\textrm{#1}}{=}}
\nc{\exact}{\mathsf{ex}}
\nc{\faithful}{\mathsf{faithful}}% faithful
\nc{\faith}{\mathsf{faithf}}
\nc{\final}{\textrm{\scriptsize{\ding{93}}}} % modified asterisk
\nc{\Funadd}{\Fun_{\amalg}}% additive functors, in the sense of Mack 1
\nc{\Funplus}{\Fun_{+}}% additive functors. Can be tweaked...
\nc{\fun}{\mathrm{fun}} % "funtorator" of pseudofunctors (or whatever it's called)
\nc{\GG}{\mathbb{G}}% the 2-category of finite groupoids `of interest'
\nc{\gpdG}{{\groupoidf_{\!\smallslash\!G}}} % the "correct" 2-category of groupoids for Mackey functors for G, i.e. the comma category of faithful functors to G
\nc{\gpdGzero}{{\groupoidf_{\!\smallslash\!G_0}}\!} % variation with G_0
\nc{\gpdfover}[1]{\groupoidf_{\!\smallslash\!#1}}
\nc{\gpd}{\groupoid}%
\nc{\gps}{\mathsf{groups}} % category of finite groups
\nc{\groconn}{\groupoid_{\mathsf{conn}}}% connected finite groupoids (auxiliary)
\nc{\groupoidf}{\groupoid{}^{\smallfaithful}}% finite groupoids with faithful morphisms
\nc{\gpdf}{\groupoidf} % short version
\nc{\groupoid}{\mathsf{gpd}}% finite groupoids
\nc{\group}{\mathsf{group}} % category of finite groups, variant...
\nc{\Gsets}{G\sset}
\nc{\HGfK}{\doublequot{H}{G}{f(K)}}%
\nc{\HGK}{\doublequot HGK}% most used
\nc{\Homcat}[1]{\Hom_{\cat #1}}
\nc{\hooklongleftarrow}{\longleftarrow\joinrel\rhook}
\nc{\hooklongrightarrow}{\lhook\joinrel\longrightarrow}
\nc{\hook}{\hookrightarrow}
\nc{\Hsets}{H\mathsf{-sets}}
\nc{\ic}{\mathsf{ic}}
\nc{\ICAdd}{\Add_{\ic}}%
\nc{\ICADD}{\ADD_{\ic}}%
\nc{\Idcat}[1]{\Id_{\cat{#1}}}
\nc{\id}{\mathrm{id}}
\nc{\Id}{\mathrm{Id}}
\nc{\ie}{{\sl i.e.}\ }
\nc{\into}{\mathop{\rightarrowtail}}
\nc{\inv}{^{-1}}
\nc{\Iout}[1]{\Ivo{\sout{#1}}}
\nc{\isocell}[1]{\undersett{ #1}{\overset{\sim}{\Ecell}}} % to be used ONLY for 2-cells in xypic diagrams [I]
\nc{\backisocell}[1]{\undersett{ #1}{\overset{\sim}{\Wcell}}} % to be used ONLY for 2-cells in xypic diagrams [I]
\nc{\Isocell}[1]{\undersett{ #1}{\overset{\sim}{\Longrightarrow}}}% ONLY inline with target and source [I]
\nc{\isoEcell}{\overset{\sim}{\Rightarrow}} % to be used ONLY in-line, with target and source [I]
\nc{\isotoo}{\stackrel{\sim}\longrightarrow}
\nc{\isoto}{\buildrel \sim\over\to}
\nc{\Ivo}[1]{{\color{OliveGreen}#1}}
\nc{\JJ}{\mathbb{J}}% old \II % the class of `faithful' guys, we can try other things.
\nc{\kk}{\Bbbk}
\nc{\KK}{\mathrm{KK}}
\nc{\leps}{{}^{\ell}\eps}
\nc{\leta}{{}^{\ell}\eta}
\nc{\loccit}{{\sl loc.\ cit.}}
\nc{\lotoo}[1]{\overset{#1}{\,\longleftarrow\,}}
\nc{\loto}[1]{\overset{#1}{\leftarrow}}
\nc{\lto}{\leftarrow}
\nc{\lun}{\mathrm{lun}} % left unitor
\nc{\Mackintro}[1]{(Mack\,\ref{Mack-#1-intro})}
\nc{\Mack}[1]{(Mack\,\ref{Mack-#1})}
\nc{\Mid}{\,\big|\,}
\nc{\MMod}{\,\text{-}\Mod}%
\nc{\PProj}{\,\text{-}\Proj}
\nc{\CComod}{\,\text{-}\Comod}
\dmo{\mods}{mod}%
\nc{\mmods}{\,\text{-}\mathrm{mod}}%
\nc{\MM}{\cat{M}}% a Mackey 2-functor
\nc{\Muniv}{\cat{M}_{\mathsf{univ}}}
\nc{\Ncell}{\rotatebox[origin=c]{0}{$\Uparrow$}} % treated as the others, for uniform size
\nc{\NEcell}{\rotatebox[origin=c]{135}{$\Downarrow$}} % North-East oriented 2-cell arrow
\nc{\NN}{\cat{N}}% another Mackey 2-functor
\nc{\noloc}{\nobreak\mspace{6mu plus 1mu}{:}\nonscript\mkern-\thinmuskip\mathpunct{}\mspace{2mu}}% mirror of \colon
\nc{\NWcell}{\rotatebox[origin=c]{-135}{$\Downarrow$}} % North-West oriented 2-cell arrow
\nc{\oEcell}[1]{\overset{\scriptstyle #1}{\Ecell}} % to be used ONLY for 2-cells in xypic diagrams [I]
\nc{\oWcell}[1]{\overset{\scriptstyle #1}{\Wcell}} % same, but for Wcell
\nc{\ointo}[1]{\overset{#1}{\rightarrowtail}}
\nc{\olto}[1]{\overset{#1}\lto}
\nc{\onto}{\mathop{\twoheadrightarrow}}
\nc{\op}{{\mathrm{op}}}
\nc{\xto}[1]{\xrightarrow{#1}}% much simpler!
\nc{\oto}[1]{\overset{#1}\to}
\nc{\Paul}[1]{{\color{Blue}#1}}
\nc{\pih}[1]{\tau_{1}#1}%
\nc{\Pout}[1]{\Paul{\sout{#1}}}
\nc{\qquadtext}[1]{\qquad\textrm{#1}\qquad}
\nc{\quadtext}[1]{\quad\textrm{#1}\quad}
\nc{\ra}{\rightarrow}
\nc{\reps}{{}^{r\!}\eps}
\nc{\restr}[1]{{|_{\scriptstyle #1}}}% to restrict a map
\nc{\reta}{{}^{r\!}\eta}
\nc{\run}{\mathrm{run}} % right unitor
\nc{\Sad}{\mathsf{Sad}}
\nc{\SAD}{\mathsf{SAD}}
\nc{\sbull}{{\scriptscriptstyle\bullet}}
\nc{\Scell}{\rotatebox[origin=c]{0}{$\Downarrow$}} % treated as the others, for uniform size
\nc{\SEcell}{\rotatebox[origin=c]{45}{$\Downarrow$}} % South-East oriented 2-cell arrow
\nc{\SET}[2]{\big\{\,#1\Mid#2\,\big\}}
\nc{\set}{\mathsf{set}} % category of finite sets
\nc{\Set}{\mathsf{Set}}% category of sets
\nc{\smallfaithful}{\mathsf{f}}% faithful
\nc{\smallslash}{{}^{\scriptscriptstyle/}}
\nc{\smat}[1]{\left(\begin{smallmatrix} #1 \end{smallmatrix}\right)}
\nc{\spanG}{{\widehat{\mathsf{gp}\,\,}\!\!\mathsf{d}}{}^\smallfaithful_{\!{}^{\scriptscriptstyle/}\!G}}% span of groupoids over G.
\nc{\Spanhat}{\textrm{\sf S}\widehat{\textrm{\sf pan}}} %
\nc{\Span}{\Spanname}% the span bicategory of a (2,1)-category
\nc{\sset}{\textrm{-}\set}
\nc{\str}{\mathsf{str}}
\nc{\SWcell}{\rotatebox[origin=c]{-45}{$\Downarrow$}} % South-West oriented 2-cell arrow
\nc{\too}{\mathop{\longrightarrow}\limits}
\nc{\tristars}{\begin{center} $ *\;*\;* $ \end{center}}
\nc{\tSpan}{\pih{\Spanname}}% 1-truncation of Span
\nc{\Unit}{\mathbb{1}}
\nc{\undersett}[1]{\underset{\scriptstyle #1}}
\nc{\un}{\mathrm{un}} % "unitor" of pseudofunctors
\nc{\vcorrect}[1]{{\vphantom{\vbox to #1em{}}}}
\nc{\Wcell}{\rotatebox[origin=c]{90}{$\Uparrow$}} % treated as the others, for uniform size
\nc{\what}[1]{\widehat{\cat{#1}}}% span category with name #1.
\nc{\xra}{\xrightarrow}

\nc{\xBur}{\mathrm{B^c}} % crossed Burnside ring
\nc{\xBurk}{ \mathrm{B}^{\mathrm{c}}_{\kk} } % crossed Burnside k-algebra
\nc{\Bur}{\mathrm{B}} % ordinary Burnside ring
\nc{\Burk}{\Bur_{\kk}} % ordinary Burnside k-algebra

\nc{\isoTo}{\overset{\sim}{\Rightarrow}}
\nc{\isoc}[3]{#1\,{\diamond}_{_{\!#3}}#2}
\nc{\Isoc}[3]{(\isoc{#1}{#2}{#3})}

\nc{\lproj}{\mathrm{Lp}} % left projection maps
\nc{\rproj}{\mathrm{Rp}} % right projection maps

\begin{document}

%------------------------------------------------------------------------------

\title{On Krull--Schmidt bicategories}

\author{Ivo Dell'Ambrogio}
\date{\today}

\address{\ \medbreak
\noindent Univ.\ Lille, CNRS, UMR 8524 - Laboratoire Paul Painlev\'e, F-59000 Lille, France}
\email{ivo.dell-ambrogio@univ-lille.fr}
\urladdr{http://math.univ-lille1.fr/$\sim$dellambr}

\begin{abstract} \normalsize
We study the existence and uniqueness of direct sum decompositions in additive bicategories. 
We find a simple definition of Krull--Schmidt bicategories, for which we prove the uniqueness of decompositions into indecomposable objects as well as a characterization in terms of splitting of idempotents and properties of 2-cell endomorphism rings.
Examples of Krull--Schmidt bicategories abound, with many arising from the various flavors of 2-dimensional linear representation theory.
\end{abstract}

\thanks{Author partially supported by Project ANR ChroK (ANR-16-CE40-0003) and Labex CEMPI (ANR-11-LABX-0007-01).}

 % 55P91, 19A22

\maketitle

%%------------------------------------------------------------------------------
\tableofcontents
%%------------------------------------------------------------------------------
%\vskip-\baselineskip\vskip-\baselineskip\vskip-\baselineskip

%------------------------------------------------------------------------------
\section{Introduction and results}
\label{sec:intro} %
%------------------------------------------------------------------------------

A Krull--Schmidt category is an additive category in which every object admits a decomposition into a direct sum of finitely many objects which are strongly indecomposable, \ie whose endomorphism rings are local. 
In fact, in a Krull--Schmidt category strongly indecomposable objects coincide with those which are indecomposable for direct sums.
Most importantly, every object is a direct sum of indecomposable ones in a unique way, up to isomorphism and permutation of the factors.
There is a nice characterization of Krull--Schmidt categories: They are precisely those additive categories in which all idempotent morphisms split and all endomorphism rings are semiperfect in the sense of Bass. See \eg~\cite{Krause15}.

The goal of this note is to prove versions of the above eminently useful facts in the setting of additive \emph{bi}categories. 
Perhaps surprisingly, the new bicategorical theory is arguably more elegant and easier to use than its classical counterpart.  

The usual 1-categorical Krull--Schmidt property holds, in particular, for categories of finite dimensional representations of groups and algebras. It is a cornerstone of representation theory which helps alleviate the pain of leaving the semisimple case for the modular case, where it becomes imperative to understand the global structure of categories of representations. 
The Krull--Schmidt property is also surprisingly subtle (\eg it holds for finite length modules over any ring but fails for artinian modules already over certain semilocal noetherian commutative rings~\cite{FHLV95} \cite{Facchini03}), and it has benefited from some serious investigation.
Similarly, as \emph{2-representation theory} (in its various directions) is beginning to venture beyond the semisimple case too, it seems wise to prepare a suitable Krull--Schmidt theory. 
We do this in Sections~\ref{sec:prelims}-\ref{sec:bc}. 

The final \Cref{sec:exas} discusses examples of Krull--Schmidt bicategories. 
They include various bimodule bicategories of rings (such as rings with noetherian center), bicategories of Mackey 2-motives in the sense of Balmer--Dell'Ambrogio (\Cref{Exa:M2M}~-- our original motivating example), 
semisimple 2-categories such as that of finite dimensional 2-vector spaces \`a~la Kapranov--Voevodsky--Neuchl (\Cref{Exa:2vect}) or 2-Hilbert spaces \`a~la Baez (\Cref{Exa:2hilb}), 
finite module categories over finite tensor categories \`a~la Etingof--Ostrik (\Cref{Cor:tens-cat}), 
the 2-category of finite dimensional 2-representations of any 2-group over any field (\Cref{Cor:2reps}),
and the 2-category of 2-representations of a finitary linear 2-category \`a~la Mazorchuk--Miemietz (\Cref{Cor:MM}).
Underlying the representation-theoretic examples is \Cref{Thm:KS-PsFun}, which categorifies the classical fact that finite dimensional representations of any algebra form a Krull--Schmidt category.
In view of all this, we expect Krull--Schmidt bicategories to enter the basic toolkit of any future modular (\ie nonsemisimple) linear 2-representation theory, where global structural questions should gain in importance relative to combinatorial ones.
\begin{center} $*\;*\;*$ \end{center}

We now present our theoretical results in details.

We say a bicategory is \emph{additive} if it is locally additive, \ie its Hom categories are additive and its composition functors are additive in both variables, and if it admits all finite direct sums of its objects; see \Cref{sec:prelims} for details. For instance additive categories, additive functors and natural transformations form a very large additive bicategory.
A nonzero object $X$ in an additive bicategory is \emph{indecomposable} if a direct sum decomposition $X\simeq X_1\oplus X_2$ necessarily implies that $X_1 \simeq 0$ or $X_2 \simeq 0$ (where $\simeq$ denotes equivalence). 
We say $X$ is \emph{strongly indecomposable} if its 2-cell endomorphism ring $\End_{\cat B(X,X)}(\Id_X)$ is a connected commutative ring (it is always commutative by the Eckmann--Hilton argument).
Here is our main result:

\begin{Thm*} [{See \Cref{Thm:KRS}}]
Suppose that an object $X$ of an additive bicategory admits two direct sum decompositions $Y_1\oplus \cdots \oplus Y_n$ and $Z_1\oplus \cdots \oplus Z_m$ into strongly indecomposable objects.
Then $n=m$ and there are equivalences $Y_k\simeq Z_{\sigma(k)}$ for all~$k$ after some permutation $\sigma \in \Sigma_n$ of the factors.
\end{Thm*}

This suggests the following definition: An additive bicategory is \emph{Krull--Schmidt} if every object is equivalent to a direct sum of finitely many strongly indecomposable objects, in the above sense. We easily deduce (\Cref{Cor:KRS}) that indecomposable and strongly indecomposable objects coincide  in a Krull--Schmidt bicategory, hence that every object admits a unique decomposition into indecomposable ones.

We also obtain a very satisfying characterization of Krull--Schmidt bicategories.
Let us call a commutative ring \emph{semiconnected} if it is a direct product of finitely many connected rings, \ie if its Zariski spectrum has finitely many topological connected components.
We also say an additive bicategory is \emph{weakly block complete}\footnote{Unfortunately, ``idempotent complete'' and ``block complete'' already have different meanings in this context.} if every decomposition $1_X = e_1 + e_2$ of the identity 2-cell $1_X \in \End_{\cat B}(\Id_X)$ as a sum of two orthogonal idempotent 2-cells is induced by a direct sum decomposition $X\simeq X_1 \oplus X_2$ at the level of objects (see details in \Cref{Def:bc}).
Then: 

\begin{Thm*} [{See \Cref{Thm:KS-char}}]
An additive bicategory is Krull--Schmidt if and only if it is weakly block complete and the 2-cell endomorphism ring $\End(\Id_X)$ of every object $X$ is semiconnected.
\end{Thm*}

Some comments are warranted. 

As far as we can see, our theory does not recover its classical counterpart because we cannot usefully view an additive category as an additive bicategory. Thus our definitions and results are just analogs rather than generalizations.

On its face, our 2-categorical definition of a strongly indecomposable object $X$ may appear to be unrelated to the original 1-categorical notion.
A direct 2-categorical analog of ``$\End(X)$ is local'' would be to require that whenever $F$ and $G$ are two 1-cells $X\to X$ which are not equivalences, then $F\oplus G\colon X\to X$ is not an equivalence either. But it turns out that the latter property is a consequence of our definition, and as soon as the category $\End (X)$ is idempotent complete (a mild hypothesis) the two become equivalent. 
In fact, all reasonable options we could think of become equivalent under the same hypothesis (\Cref{Cor:strongly-indec-X}). 

Obviously, in an additive 1-category endomorphism rings are rarely commutative and there are no 2-cell endomorphism rings. 
This might explain why passing to a higher-categorical setting actually appears to make indecomposables and the uniqueness of decompositions \emph{easier} to handle---insofar as semiconnected rings are easier than semiperfect ones.
It also makes examples easier to recognize, \eg we do not feel any pressing need to discuss higher versions of bi-chain conditions, finite length objects, etc. In order to recognize a Krull--Schmidt bicategory using our characterization, it suffices to ensure weak block completeness (which can always be implemented by \cite[Thm.\,A.7.23]{BalmerDellAmbrogio20}) and to have a reasonable grasp of the 2-cell Homs (\eg it suffices that they are not unreasonably large).

We thought it instructive to give two different proofs of the uniqueness theorem. 
The first proof (in \Cref{sec:KS}), which seems not to have any 1-categorical analog, easily follows from two facts. Firstly, that a direct summand $Y$ of $X$ is determined up to equivalence by the idempotent element of $\End(\Id_X)$ corresponding to~$1_{Y}$; note that, by definition, the summand is strongly indecomposable precisely when the idempotent is primitive. Secondly, that in a commutative ring there cannot be two different decompositions of~$1$ as a sum of orthogonal primitive idempotents. 
The second proof (in \Cref{sec:bc}) is a calque of the usual 1-categorical proof, and proceeds by recursively comparing summands of two decompositions of the same object via the given equivalences. 
Both proofs are easy in principle but require a close inspection of the bicategorical notion of ``direct summand'', which we carry out in \Cref{sec:dirsums}.
To wit, in weakly block complete bicategories we characterize a direct summand $Y$ of $X$ in a few equivalent ways: as a special kind of ambijunction $Y \leftrightarrows X$, or (forgetting the object~$Y$) as a special kind of Frobenius monad on~$X$, or quite magically (forgetting the 1-cell $X\to X$) as an idempotent in $\End(\Id_X)$.

\begin{Ack*}
I am grateful to Alexis Virelizier for pointing me to \cite{DouglasReutter18pp} and to Paul Balmer and an anonymous referee for useful comments.
\end{Ack*}

%------------------------------------------------------------------------------
\section{Preliminaries on direct sums}
\label{sec:prelims} %
%------------------------------------------------------------------------------

Most of these preliminaries are also covered in \cite[\S A.7]{BalmerDellAmbrogio20}, where the reader may find more details and context.

\begin{Conv}
We will use the standard terminology for bicategories, as recalled \eg in \cite[App.\,A]{BalmerDellAmbrogio20}.
However, whenever convenient and without further mention we will pretend our bicategories are strict 2-categories in order to simplify calculations and to more easily talk about (internal) monads; this is justified by the coherence theorem for bicategories (\cite[Ch.\,8]{JohnsonYau21}).
If $\cat B$ is a bicategory, we write $\cat B(X,Y)$ for its Hom categories and $\End_{\cat B(X,Y)}(F)$ or just $\End_\cat B(F)$ or $\End(F)$ for the 2-cell endomorphism set of a 1-cell $F\colon X\to Y$. Internally to a bicategory, $\cong$ and $\simeq$ stand for isomorphism and equivalence, respectively; $\Id_X$ and $\id_F$ for identity 1-cells and identity 2-cells, respectively; and also $1_X$ or just $1$ is short for~$\id_{\Id_X}$.
Products, coproducts, initial and final objects in a bicategory are understood in the sense of \emph{pseudo}\-(co)\-limits (a.k.a.\ bilimits \cite[\S5]{JohnsonYau21}), rather than strict ones; in particular, they are only unique up to equivalence rather than isomorphism.
We mostly ignore set-theoretical size issues as irrelevant to this topic.
\end{Conv}

\begin{Def}
A bicategory $\cat B$ is \emph{locally additive} if each Hom category $\cat B(X,Y)$ is additive and each composition functor $\cat B(Y,Z)\times \cat B(X,Y)\to \cat B(X,Z)$ is additive in both variables.
\end{Def}

\begin{Def} \label{Def:dir-sum}
Let $\cat B$ be a locally additive bicategory.
A \emph{direct sum} of two objects $X_1$ and $X_2$ is a diagram of 1-cells of $\cat B$ of the form 
\begin{equation} \label{eq:dir-sum-diag}
\xymatrix{
X_1 
 \ar@<.5ex>[r]^-{I_1} & 
X 
 \ar@<.5ex>[l]^-{P_1} 
 \ar@<-.5ex>[r]_-{P_2} &
X_2 
 \ar@<-.5ex>[l]_-{I_2}
}
\end{equation}
such that there exist isomorphisms 
$P_iI_i \cong \Id_{X_i}$ and $P_iI_j \cong 0$ (for $i\neq j$) and $I_1P_1 \oplus I_2 P_2 \cong \Id_X$, the latter two being a zero object and a  direct sum in the additive categories $\Hom_\cat B(X_j,X_i)$ and $\End_\cat B(X)$, respectively.
As usual, we will write $X_1 \oplus X_2$ to indicate the object $X$ or even the whole diagram; 
direct sums of $n\geq 2$ objects are defined similarly; and the direct sum of no objects is by definition a zero objet of~$\cat B$, \ie one which is both initial and final.
\end{Def}

\begin{Def} \label{Def:add-bicat}
We say that a bicategory is \emph{additive} if it is locally additive and if it admits arbitrary finite direct sums of its objects.
\end{Def}

\begin{Exa} \label{Exa:ADD}
The (very large) 2-category $\ADD$ of (non-necessarily small) additive categories, additive functors and natural transformations is an additive bicategory. 
Direct sums are provided by the (strict) finite Cartesian products of categories, with their projection functors $P_i\colon \cat A_1 \times \cat A_2\to \cat A_i$ and with inclusion functors defined by $I_1\colon A_1\mapsto (A_1,0)$ and $I_2\colon A_2 \mapsto (0,A_2)$.

\end{Exa}

\begin{Def} \label{Def:indec}
An object $X$ of an additive bicategory $\cat B$ is \emph{indecomposable} if whenever there is an equivalence $X\simeq X_1 \oplus X_2$ then $X_1\simeq 0$ or $X_2 \simeq 0$.
\end{Def}

\begin{Rem}  \label{Rem:matrix-not}
One easily checks that any direct sum as in \Cref{Def:dir-sum} is also a product $(X,P_1,P_2)$ and a coproduct $(X,I_1, I_2)$ of $X_1$ and~$X_2$ (and similarly with $n\geq 2$ factors). 
It follows that direct sums are unique up to equivalence. 
Also, every direct sum decomposition $X\simeq X_1\oplus X_2$ induces equivalences of additive categories
\begin{equation} \label{eq:decomp-cats}
\cat B(Y, X) \simeq \cat B(Y,X_1) \times \cat B(Y, X_2)
\quad \textrm{ and } \quad
\cat B(X,Y) \simeq \cat B(X_1,Y) \times \cat B(X_2,Y)
\end{equation}
for all objects~$Y$. 
It is therefore possible to use matrix notation in the usual way for 1-cells as well as 2-cells between two directs sums, just as one does inside each additive Hom category.
In particular (\cf \cite[Rem.\,A.7.15]{BalmerDellAmbrogio20}), we may decompose the category $\End_\cat B(X)$ into four factors.
Then $\Id_X$ corresponds to the diagonal matrix $\mathrm{diag}(\Id_{X_1}, \Id_{X_2})$ and its endomorphism ring $\End_{\cat B(X,X)}(\Id_X)$ is also diagonal (\cf \Cref{Lem:comm-matrices} if necessary), so that we get an isomorphism of commutative rings
\[
\End_{\cat B}(\Id_X) \cong \End_\cat B(\Id_{X_1}) \times \End_\cat B(\Id_{X_2}).
\]
The latter is equivalent to a decomposition of $1_X= \id_{\Id_X}$ as a sum $e_1+e_2$ of two orthogonal idempotents, with $e_i$ corresponding to~$1_{X_i}$.
Given such a decomposition of~$1_X$, however, nothing guarantees that it arises as above from a direct sum decomposition $X\simeq X_1\oplus X_2$. Whence:
\end{Rem}

\begin{Def} \label{Def:bc}
An additive bicategory is \emph{weakly block complete} if for every object $X$ and every idempotent element $e=e^2$ of the ring $\End_\cat B(\Id_X)$ there exists a direct sum decomposition $X\simeq X_1\oplus X_2$ under which the 2-cells $e$ and $1_X-e$ correspond to $[\,{}^1_0\,{}^0_0\,]$ and $[\,{}^0_0\,{}^0_1\,]$, respectively, as in \Cref{Rem:matrix-not}.
As in \cite{BalmerDellAmbrogio20}, we say that $\cat B$ is \emph{block complete} if moreover every Hom category of $\cat B$ is idempotent complete.
\end{Def}

\begin{Thm} [{See \cite[A.7.23--24]{BalmerDellAmbrogio20}}]
\label{Thm:bc}
For every additive bicategory $\cat B$, there exists a block complete bicategory $\cat B^\flat$ and a 2-fully faithful pseudofunctor $\cat B\hookrightarrow \cat B^\flat$ which is 2-universal among additive pseudofunctors to block complete bicategories. 
\end{Thm}

\begin{Rem} \label{Rem:wbc}
Using \Cref{Thm:bc}, we can also produce a similarly universal \emph{weak block completion} $\cat B^{w\flat}$ for any additive bicategory~$\cat B$: 
It suffices to take the smallest 2-full sub-bicategory of $\cat B^\flat$ which contains (the image of)~$\cat B$ and is weakly block complete. It will have the same objects as $\cat B^\flat$ but may lack some of its 1-cells.
\end{Rem}

In \Cref{sec:exas} we will need the fact that block completeness and weak block completeness are inherited pointwise, in the sense that bicategories of pseudofunctors inherit these properties from their target bicategory. More precisely:

\begin{Ter} \label{Ter:PsFun}
Let $\kk$ be a commutative ring.  
A bicategory is \emph{$\kk$-linear} if its Hom categories and composition functors are $\kk$-linear (\ie enriched in $\kk$-modules). 
For instance, a locally additive bicategory is a $\mathbb Z$-linear bicategory which also admits finite direct sums of parallel 1-cells.
A pseudofunctor $\cat F\colon \cat B \to \cat C$ between $\kk$-linear bicategories is \emph{$\kk$-linear} if each functor $\cat F\colon \cat B(X,Y)\to \cat C(\cat FX, \cat FY)$ is $\kk$-linear, \ie preserves $\kk$-linear combinations of 2-cells.
Note that a $\kk$-linear pseudofunctor preserves whatever direct sums of 1-cells and objects exist in its source, as the latter are equationally defined (\cf \cite[Prop.\,A.7.14]{BalmerDellAmbrogio20}).
We will denote by
\[
\PsFun_\kk(\cat B, \cat C)
\]
the bicategory of $\kk$-linear pseudofunctors $\cat B\to \cat C$, pseudonatural transformations between them, and modifications;
see \eg \cite[\S A.1]{BalmerDellAmbrogio20} for details.
\end{Ter}

\begin{Prop} \label{Prop:levelwise-wbc}
If $\cat C$ is (weakly) block complete then so is $\PsFun_\kk(\cat B, \cat C)$.
\end{Prop}

\begin{proof}
First note that $\PsFun_\kk(\cat B, \cat C)$ inherits a pointwise $\kk$-linear structure. Namely, suppose $\cat F_1,\cat F_2\colon \cat B\to \cat C$ are pseudofunctors, $t,s\colon \cat F_1\Rightarrow \cat F_2$ pseudonatural transformations, $M,M'\colon t\Rrightarrow s$ modifications, and $\lambda \in \kk$ a scalar. 
Then there is a modification $M + \lambda M'\colon t\Rrightarrow s$ whose component at each object $X\in \cat B$ is defined by setting $(M + \lambda M')_X:= M_X + \lambda M'_X $ in the $\kk$-module $\Hom_{\cat B(\cat F_1 X, \cat F_2X)}(t_X , s_X)$. 
Similarly, $\PsFun_\kk(\cat B, \cat C)$ inherits direct sums of 1-cells and objects from~$\cat C$: 
The components in $\cat C$ of the transformation $t \oplus s\colon \cat F_1 \Rightarrow \cat F_2$ are
\[
(t\oplus s)_X := t_X \oplus s_X \colon \cat F_1X \to \cat F_2X
\]
for all objects $X\in \cat B$ and the diagonal
\[
(t\oplus s)_F:= t_F \oplus s_F = \left(\, {}^{t_F}_0 \,\; {}^0_{s_F} \right) \colon \cat F_2F\circ (t\oplus s)_X \overset{\sim}{\Longrightarrow} (t \oplus s)_Y\circ \cat F_1 F
\]
for all 1-cells $F\colon X\to Y$ of~$\cat B$ (for the latter, we use that horizontal composition in $\cat C$ preserves direct sums of 1-cells), with the evident direct sum structure maps.

For the direct sum $\cat F_1 \oplus \cat F_2$ of pseudofunctors,  set $(\cat F_1 \oplus \cat F_2)(X):= \cat F_1 X \oplus \cat F_2 X$ on objects, with diagonal component $\kk$-linear functors (at each pair $X,Y\in \Obj\cat B$)
\[
\cat F_1 \oplus \cat F_2:= \left(\, {}^{\cat F_{1}}_0 \,\; {}^0_{\cat F_{2}} \right) \colon \cat B(X,Y) \longrightarrow \cat C(\cat F_1 X \oplus \cat F_2 X, \cat F_1 Y \oplus \cat F_2 Y),
\]
where the right-hand side is decomposed into four categories as in \Cref{Rem:matrix-not}. 

It remains to see that $\PsFun_\kk(\cat B, \cat C)$ also inherits the splitting of 1-cells and objects.
Suppose $M= M^2\colon t \Rrightarrow t$ is an idempotent modification. 
Then each component $M_X\colon t_X\Rightarrow t_X$ is an idempotent 2-cell of~$\cat C$, giving rise to a decomposition $t_X \cong \img(M_X) \oplus \img (\id_{t_X} - M_X)$ as soon as that idempotent splits in  $\cat C(\cat F_1X, \cat F_2X)$.
We can then assemble these splittings into an isomorphism $t \cong \img(M) \oplus \img (\id_t - M)$ in the additive category $\PsNat(\cat F_1, \cat F_2)$ of transformations and modifications.
Thus if $\cat C$ is locally idempotent complete so is $\PsFun_\kk(\cat B, \cat C)$.
If $\cat C$ is weakly block complete, in the special case when $t=\Id_{\cat F}$ is the identity transformation of $\cat F:= \cat F_1 = \cat F_2\colon \cat B\to \cat C$ we also have decompositions $\cat FX \simeq Y_{X} \oplus Z_{X}$ of objects of~$\cat C$ with $1_{Y_{X}}$ and $1_{Z_{X}}$ corresponding under these equivalences to the identity 2-cells of $\img(M_X)$ and $\img(1_{\cat FX} - M_X)$, respectively. 
Any choice of such adjoint equivalences assembles into the required decomposition $\Id_\cat F \simeq \img(M)\oplus \img (1_\cat F - M)$ in $\PsNat(\cat F_1, \cat F_2)$. (If necessary, to see why the latter assertion is true one may use the characterizations of direct summands treated in the next section.)
\end{proof}

\begin{Rem}
In \cite{DouglasReutter18pp} a notion of ``idempotent completion'' for linear 2-categories is used which in general is stronger than our  block completion: It requires every Hom category to be idempotent complete and every separable monad to split (\ie to arise from a separable adjunction). 
\end{Rem}

%------------------------------------------------------------------------------
\section{Characterizations of direct summands}
\label{sec:dirsums} %
%------------------------------------------------------------------------------

We work throughout in a locally additive bicategory.
Note that, when referring to a direct sum $X_1\oplus X_2$ as in~\eqref{eq:dir-sum-diag}, we were implicitly assuming a choice was made for its four projection and injection 1-cells, but not for the 2-isomorphisms implementing their relations.
We can do better:

\begin{Prop} \label{Prop:adj-sums}
Consider any direct sum $X= X_1 \oplus X_2$  as in~\eqref{eq:dir-sum-diag}.
Then the 2-cells implementing the direct sum relations can be chosen so that they simultaneously form four adjunctions $P_i \dashv I_i$ and $I_i\dashv P_i$ ($i= 1,2$).
More precisely,
write
\[
\xymatrix@L=5pt{
P_i I_i 
  \ar@<1ex>@{=>}[r]^-{\varepsilon_i}
  \ar@{}[r]|-{\sim} & 
\Id_{X_i}
  \ar@<1ex>@{=>}[l]^-{\overline{\eta}_i}
}
\quad
 \textrm{ and } 
\quad
\xymatrix@L=5pt{
I_1P_1 \oplus I_2 P_2
  \ar@<1ex>@{=>}[r]^-{
   \left[ \begin{array}{cc} \!\!\!\scriptstyle{\overline{\varepsilon}_1} \!\!&\!\! \scriptstyle{\overline{\varepsilon}_2} \!\!\!\!\! \end{array} \right] 
  }
  \ar@{}[r]|-{\sim} & 
\Id_X
  \ar@<1ex>@{=>}[l]^-{
  \left[
  \begin{array}{c}
  \!\!\!\scriptstyle{\eta_1}\!\!\!\!\! \\
  \!\!\!\scriptstyle{\eta_2}\!\!\!\!\!
  \end{array}
  \right]
  } ,
}
\]
for the 2-cell components; they can be chosen so that, besides satisfying the relations
\begin{equation} \label{eq:dir-sum-eqs}
\varepsilon_i = (\overline{\eta}_i)^{-1}
\quad\quad\quad
\eta_i \overline{\varepsilon}_j = 
	\left\{
	\begin{array}{ll}
	\id_{I_iP_i}  & \textrm{ if } i=j \\
	0 & \textrm{ if } i\neq j
	\end{array}
	\right.
\quad\quad\quad
\overline{\varepsilon}_1\eta_1 + \overline{\varepsilon}_2\eta_2 = 1_X 
\end{equation}
they also partake in the adjunctions (for $i=1,2$)
\[
(P_i \dashv I_i, \eta_i, \varepsilon_i)
\quad \textrm{ and } \quad
(I_i \dashv P_i, \overline{\eta}_i, \overline{\varepsilon}_i)
\]
as units or counits.
 The analogous statement holds for sums of $n\geq 2$ objects.
\end{Prop}

\begin{Def} \label{Def:adj-sum}
An \emph{adjoint direct sum} in a locally additive bicategory is the data of a direct sum diagram together with a choice of 2-cells as in \Cref{Prop:adj-sums}.
\end{Def}

\begin{Rem} \label{Rem:almost-isos}
Consider any direct sum diagram \eqref{eq:dir-sum-diag}, say with 2-cell components $\beta_i\colon P_iI_i\overset{\sim}{\Rightarrow} \Id_{X_i}$ whose inverse we will write~$\overline{\alpha}_i$, and $[\overline{\beta}_1 \;\; \overline{\beta}_2]\colon I_1P_1 \oplus I_2P_2\overset{\sim}{\Rightarrow} \Id_X$ with inverse ${}^t[\alpha_1 \;\; \alpha_2]$. 
These 2-cells will satisfy the direct sum relations as in~\eqref{eq:dir-sum-eqs}, of course. 
Moreover, not only are $\beta_i$ and $\overline{\alpha}_i$ mutually inverse, but we also have
\[
P_i \overline{\beta}_i = (P_i \alpha_i)^{-1} \colon P_i \overset{\sim}{\Rightarrow} P_i
\quad \textrm{ and } \quad
\overline{\beta}_i I_i = (\alpha_i I_i)^{-1}\colon I_i \overset{\sim}{\Rightarrow} I_i.
\]
Indeed, these 2-cell pairs are already mutual inverses on one side (even before applying $P_i$ or~$I_i$) by one of the defining relations.
To see that they are inverses on the other side too, it suffices to whisker the relation $\overline{\beta}_1\alpha_1 + \overline{\beta}_2\alpha_2 = 1_X$ by $P_i$ or~$I_i$, respectively, and use that $P_iI_j \cong 0$ for $i\neq j$.
We will often use this observation.
\end{Rem}

\begin{proof}[Proof of \Cref{Prop:adj-sums}]
This is a sharper version of \cite[Prop.\,1.1.3]{DouglasReutter18pp}, and the proof is similar.
Start with a direct sum diagram with some choice of 2-cells, with notation as in \Cref{Rem:almost-isos}.
For each~$i$, consider the automorphism 
\[
\varphi_i \colon 
\xymatrix@C=14pt{ 
I_iP_i
  \ar@{=>}[rr]^-{I_iP_i \alpha_i}_-\sim &&
I_iP_iI_iP_i 
  \ar@{=>}[rr]^-{I_i \beta_i P_i}_-\sim  &&
I_iP_i
}
\]
of~$I_iP_i$, whose inverse is $\varphi_i^{-1}= (I_iP_i \overline{\beta}_i)(I_i \overline{\alpha}_i P_i)$ by \Cref{Rem:almost-isos}.
We claim that the following slightly corrected 2-cells 
\[
\eta_i := \varphi_i^{-1} \alpha_i
\quad\quad
\varepsilon_i := \beta_i
\quad\quad
\overline{\eta}_i := \overline{\alpha}_i  
\quad\quad
\overline{\varepsilon}_i := \overline{\beta}_i \varphi_i
\]
satisfy the required conditions.
They again form a direct sum for the same 1-cells, indeed the relations \eqref{eq:dir-sum-eqs} are immediately checked. 
It remains to see that we have two adjunctions $(P_i \dashv I_i, \eta_i , \varepsilon_i)$ and $(I_i \dashv P_i, \overline{\eta}_i , \overline{\varepsilon}_i)$.
For the first adjunction, \Cref{Rem:almost-isos} and the commutative diagram on the left (where we dropped the~$i$'s)
\[
\xymatrix@C=14pt@R=16pt{
P
  \ar@/^3ex/@{=>}[rrrrd]^-{P\eta}
  \ar@{=>}[d]_{P\alpha} && &&  \\
PIP 
  \ar@{=>}[rr]^-{PI\overline{\alpha} P}
  \ar@{=>}[d]_{\beta P} &&
PIPIP
   \ar@{=>}[rr]^(.43){PIP\overline{\beta}}
   \ar@{=>}[d]_{\beta PIP} &&
PIP
  \ar@{=>}[d]_-{\varepsilon P} &&  \\
P 
 \ar@{=>}[rr]^-{\overline{\alpha} P} &&
PIP
  \ar@{=>}[rr]^-{P\overline{\beta}}  &&
P 
}
\!\!\!\!\!\!\!\!\!\! %\quad
\xymatrix@C=14pt@R=16pt{
P
  \ar@{=>}[rr]^-{\overline{\eta} P}
  \ar@{=>}[d]_{P\alpha} &&
PIP 
  \ar@/^3ex/@{=>}[ddrr]^{P \overline{\varepsilon}} 
  \ar@{=>}[d]_{PIP \alpha}  && \\
PIP 
  \ar@{=>}[rr]^-{\overline{\alpha} PIP}
  \ar@{=>}[d]_{\beta P} &&
PIPIP
  \ar@{=>}[d]_{PI\beta P} && \\
P 
  \ar@{=>}[rr]^-{\overline{\alpha} P} &&
PIP
  \ar@{=>}[rr]^-{P \overline{\beta}} && 
P
}
\]
verify the triangular identity $(\varepsilon_i P_i)(P_i \eta_i)=\id_{P_i}$. 
This also implies $(I_i \varepsilon_i)(\eta_i I_i) = (I_i \varepsilon_i)(I_i \varepsilon_i P_iI_i)(I_i P_i \eta_i I_i)(\eta_i I_i) = (I_i \varepsilon_i)(\eta_i I_i)(I_i \varepsilon_i)(\eta_i I_i)$ is an idempotent $I_i\Rightarrow I_i$. 
As the latter 2-cell is invertible (by \Cref{Rem:almost-isos}) it must be equal to~$\id_{I_i}$, and we get the second triangular identity.
For the second adjunction, the above diagram on the right displays (again thanks to \Cref{Rem:almost-isos}) the triangular identity $(P_i \overline{\varepsilon}_i)(\overline{\eta}_i P_i)= \id_{P_i} $, and the other one follows by the precise same argument as before.
\end{proof}

\begin{Rem} \label{Rem:yoneda-proof}
There is a far less elementary but perhaps more suggestive proof of \Cref{Prop:adj-sums}, which goes as follows. 
Firstly note that the 2-category $\ADD$ of additive categories has a canonical choice of \emph{adjoint} direct sums: Just take Cartesian products, with the projections $P_i$ and injections $I_i$ as in \Cref{Exa:ADD} and the evident 2-cells (involving solely identity and zero natural transformations). Indeed, in this case we even have $P_iI_i = \Id$ and all required triangular identities are trivially true. 
Secondly, for any locally additive (or just $\mathbb Z$-linear) bicategory~$\cat B$, the bicategory $\PsFun_\mathbb Z(\cat B^\op,\ADD)$ of additive pseudofunctors (\Cref{Ter:PsFun}) inherits finite adjoint direct sums  from $\ADD$, by the very same level-wise construction as in the proof of \Cref{Prop:levelwise-wbc}.
Thirdly, observe that adjoint direct sums (or even just their 2-cells) can be transported along (adjoint) equivalences. 
Finally, recall (\eg from \cite[Rem.\,A.7.25]{BalmerDellAmbrogio20}) that for any locally additive bicategory~$\cat B$ there is a Yoneda embedding $\cat B\to \PsFun_\mathbb Z (\cat B^\op, \ADD)$ which is an additive pseudofunctor and a biequivalence $F\colon \cat B\overset{\sim}{\to}\cat B'$ onto its full image~$\cat B' \subset \PsFun_\mathbb Z (\cat B^\op, \ADD)$. 
Using $F$ and a pseudo\-inverse~$F^{-1}$, any direct sum diagram that exists in $\cat B$ can therefore be corrected into an adjoint sum by the previous remarks.
\end{Rem}

In the following, we will tacitly assume all our direct sums to be adjoint in the sense of \Cref{Def:adj-sum}.
Let us now study a single summand at a time.

\begin{Def} \label{Ter:dirsum}
In a locally additive bicategory, a \emph{splitting datum} of an object~$X$ consists of an object~$Y$, two 1-cells $I\colon Y\to X$ and $P\colon X\to Y$ and two adjunctions
\[
(P \dashv I, \eta, \varepsilon)
\quad \textrm{ and } \quad
(I \dashv P, \overline{\eta}, \overline{\varepsilon})
\]
such that $\varepsilon = \overline{\eta}^{-1}$ and $\eta \overline{\varepsilon}=\id_{IP}$.
A splitting datum for $X$ is a \emph{direct summand of~$X$} if it appears as one of the ambijunctions $I_i\dashv P_i\dashv I_i\colon X_i\leftrightarrows X$ for some decomposition into an adjoint direct sum $X\simeq X_1\oplus \cdots \oplus X_n$ as in \Cref{Def:adj-sum}. 
A \emph{morphism} of splitting data $(Y,I,P,\eta,\varepsilon,\overline{\eta},\overline{\varepsilon})$ and $(Y',I',P',\eta',\varepsilon',\overline{\eta}',\overline{\varepsilon}')$ is a 1-cell $F\colon Y\to Y'$ together with a 2-isomorphism $\theta_F\colon FP\cong P'$, and a \emph{2-morphism} $(F_1,\theta_1)\Rightarrow (F_2,\theta_2)$ is a 2-cell $\alpha \colon F_1\Rightarrow F_2$ such that $\theta_2 (\alpha P)= \theta_1$.
(In other words, forgetting some structure, here we are considering splitting data as objects of the comma bicategory of $\cat B$ under~$X$.)
We thus obtain a notion of \emph{equivalence} of splitting data for~$X$.
\end{Def}

In the following series of remarks, we mention several ways one can slice and repackage the information contained in a splitting datum. 

\begin{Rem}  \label{Rem:refl}
An adjunction in a bicategory is a \emph{reflection} if its counit  is invertible and a \emph{coreflection} if its unit  is.
Thus the adjunctions $(P \dashv I, \eta, \varepsilon)$ and $(I \dashv P, \overline{\eta}, \overline{\varepsilon})$ in \Cref{Ter:dirsum} display $Y$ as a reflection, respectively a coreflection, of~$X$. 
Indeed the counit $\varepsilon$ and the unit $\overline{\eta}$ are mutually inverse, and we also have $\eta\overline{\varepsilon} = \id_{IP}$.
\end{Rem}

\begin{Rem} \label{Rem:monad}
Recall that a monad or comonad on an object $X$ in a bicategory is \emph{idempotent} if its (co)multiplication is invertible.
Thus writing $E := I P\colon X\to X$, we see that a splitting datum induces both a monad and a comonad on~$X$
\[
\big( E,  \xymatrix{ E^2 \ar@{=>}[r]_-\sim^-{I \varepsilon P} & E } , \xymatrix{ \Id_X \ar@{=>}[r]^-{\eta} & E } \big) 
\quad \textrm{ and } \quad
\big( E, \xymatrix{ E \ar@{=>}[r]_-\sim^-{I \overline{\eta} P } & E^2 } , \xymatrix{ E \ar@{=>}[r]^-{\overline{\varepsilon}} & \Id_X} \big)
\]
with the same underlying 1-cell~$E$ which are idempotent and satisfy $\eta\overline{\varepsilon}=\id_{E}$.
\end{Rem}

\begin{Rem} \label{Rem:loc}
An idempotent monad $(T\colon X\to X, \mu\colon T^2\overset{\sim}{\Rightarrow} T, \iota\colon \Id_X \Rightarrow T)$ on an object~$X$ is the same thing as a \emph{localization} of~$X$, namely a pair $(T,\iota\colon \Id_X\Rightarrow T)$ such that $T\eta$ and $\eta T\colon T \Rightarrow T^2$ are equal and invertible (their common inverse then recovers the multiplication~$\mu$).
Dually, an idempotent comonad is the same as a \emph{colocalization}. 
Thus \Cref{Rem:monad} says that a splitting datum amounts to a  pair
\[
(E , \eta \colon \Id_X \Rightarrow E ) 
\quad \textrm{ and } \quad
(E , \overline{\varepsilon} \colon E \Rightarrow \Id_X)
\]
of a localization and a colocalization of~$X$, respectively, with the same underlying 1-cell~$E= IP$ and such that $\overline{\varepsilon}$ is a section of~$\eta$.
\end{Rem}

\begin{Rem} \label{Rem:special-Frob}
The monad and comonad of \Cref{Rem:monad} form an \emph{idempotent special Frobenius monad} $(T, \mu, \iota, \delta, \epsilon)$ on~$X$ such that~$\iota \epsilon = \id_{T}$. 
Indeed, the latter equation is $\eta \overline{\varepsilon} = \id_{E}$ again; 
the monad and comonad automatically form a Frobenius monad as they originate from an ambijunction (see \eg \cite[Prop.\,7.4]{DellAmbrogio21pp}); 
 ``idempotent'' refers to multiplication~$\mu$ and comultiplication~$\delta$ being invertible;
 and ``special'' refers to the equation $\mu\delta=\id_E$, which holds since $\mu =(I \varepsilon P) = (I \overline{\eta} P)^{-1} = \delta^{-1}$.
\end{Rem}

\begin{Def} \label{Def:assoc-idemp}
For any splitting datum $(Y,I,P,\eta,\varepsilon,\overline{\eta},\overline{\varepsilon})$ for~$X$, the composite 2-cell $e := \overline{\varepsilon} \eta\colon \Id_X \Rightarrow \Id_X$ is an idempotent element of the commutative ring $\End_\cat B(\Id_X)$.
We call it \emph{the idempotent associated} to the given splitting datum.
Note that the idempotents $e_i$ associated to the direct summands of a direct sum decomposition $X\simeq X_1 \oplus \cdots \oplus X_n$ are orthogonal and their sum is~$1_X$.
\end{Def}

Our last goal for this section is to show that a splitting datum is determined up to equivalence by its associated idempotent (\Cref{Prop:idemp-char}).
Let us first recall that the ``image'' of a localization is similarly uniquely, essentially because it is both the Kleisli object and the Eilenberg--Moore object of the associated monad.

\begin{Lem} \label{Lem:EM-UP}
Let $I\colon Y \leftrightarrows X \!:  P$ and $I'\colon Y' \leftrightarrows X \!:  P'$ be two reflections of an object $X$ in a bicategory~$\cat B$, as in \Cref{Rem:refl}. Let $(T,\iota)$ and $(T',\iota')$ be the associated localizations of~$X$, as in  \Cref{Rem:loc}.
There is an isomorphism $T\cong T'$ identifying $\iota$ with~$\iota'$ if and only if there is an equivalence $Y\simeq Y'$ identifying the two adjunctions.
\end{Lem}

\begin{proof}
Write $\eta,\varepsilon$ and $\eta',\varepsilon'$ for the unit and counit of $P\dashv I$ and $P'\dashv I'$, respectively.
By hypothesis $\varepsilon$ and $\varepsilon'$ are invertible, and suppose that $(T,\iota)\cong (T',\iota')$, \ie that there is an invertible 2-cell $\theta \colon IP \cong I'P'$ such that $\theta \eta = \eta' \colon \Id_X \Rightarrow I'P'$.
We claim that  the composites $F:= P'I\colon Y\to Y'$ and $G:= PI' \colon Y'\to Y$ are part of an adjoint equivalence $Y\simeq Y'$ in~$\cat B$.

Indeed, from the triangular identity $(\varepsilon P)(P \eta)=\id_P$ and the invertibility of $\varepsilon$ we deduce that $P\eta \colon P\overset{\sim}{\Rightarrow} PIP$ is invertible. 
From $(P\theta) (P\eta) = P(\theta \eta) = P\eta'$ we further deduce the invertibility of $P\eta' \colon P \Rightarrow PI'P'$, hence of the composite
\[
\alpha\colon 
\xymatrix{
\Id_Y
 \ar@{=>}[r]^-{\varepsilon^{-1}} &
PI
 \ar@{=>}[r]^-{P \eta' I} &
PI'P'I = GF.
}
\]
By symmetry, we similarly obtain the isomorphism
\[
\beta\colon
\xymatrix{
\Id_{Y'}
 \ar@{=>}[r]^-{{\varepsilon'}^{-1}} &
P'I'
 \ar@{=>}[r]^-{P' \eta I'} &
P'IPI' = FG.
}
\]
The commutative diagram of isomorphisms
\[
\xymatrix{
P'I 
  \ar@/^6ex/@{=>}[rr]^-{F\alpha}
  \ar@/_3ex/@{=>}[dr]_-{\id} & 
P'I PI 
  \ar@{=>}[l]_-{\;\; P'I \varepsilon}
  \ar@{=>}[r]^-{\;\; P'IP \eta' I} & 
P'IPI'P'I 
  \ar@/^6ex/@{=>}[dd]^-{\beta^{-1}F} \\
& 
P'I
  \ar@{=>}[u]^{P'\eta I}
  \ar@{=>}[r]_-{\;\; P'\eta' I}
  \ar@/_3ex/@{=>}[dr]_-{\id} & 
P'I'P'I
  \ar@{=>}[u]^{P'\eta I' P'I}
  \ar@{=>}[d]_{\varepsilon ' P'I} \\
& & P'I
}
\]
(which uses a triangular identity of $P\dashv I$ and one of $P'\dashv I'$) displays one of the triangular identities for an adjunction $(F \dashv G,\alpha, \beta^{-1})$; the other one is similar. 

We must still verify that this adjoint equivalence $F\colon Y \simeq Y'\!: G$ matches the given adjunctions $P\dashv I$ and $P'\dashv I'$.
Indeed, we have the isomorphism $P'\varepsilon \colon FP = P'I P \overset{\sim}{\Rightarrow} P'$, and its mate under adjunction is an isomorphism $I'\cong IG$.
(One also checks immediately that $F\dashv G$ is in fact an adjoint equivalence both in the comma bicategory of~$\cat B$ under~$X$ and that over~$X$.)

The converse implication is obvious.
\end{proof}

\begin{Prop} \label{Prop:idemp-char}
Let $X$ be an object of an additive bicategory.
Let $I\colon Y \leftrightarrows X \! : P$ and $J\colon Z \leftrightarrows X \! : Q$ be two splitting data for~$X$ (\Cref{Ter:dirsum}). 
Then the two are equivalent, as splitting data for~$X$, if and only if their associated idempotent elements of $\End(\Id_X)$ are equal (\Cref{Def:assoc-idemp}). 
\end{Prop}

\begin{proof}
Let us introduce the following notation for the units and counits of the four adjunctions $I \dashv P \dashv I$ and $J \dashv Q \dashv J$: 
\begin{align*}
\eta_Y \colon \Id_X\Rightarrow IP 
\quad \quad
\varepsilon_Y \colon PI \overset{\sim}{\Rightarrow} \Id_Y
\quad \quad
\overline{\eta}_Y \colon \Id_Y \overset{\sim}{\Rightarrow} PI
\quad\quad
\overline{\varepsilon}_Y \colon IP \Rightarrow \Id_X \\
\eta_Z \colon \Id_X\Rightarrow JQ 
\quad \quad
\varepsilon_Z \colon QJ \overset{\sim}{\Rightarrow} \Id_Z
\quad \quad
\overline{\eta}_Z \colon \Id_Z \overset{\sim}{\Rightarrow} QJ
\quad\quad
\overline{\varepsilon}_Z \colon JQ \Rightarrow \Id_X
\end{align*}
Recall that $\varepsilon_Y = \overline{\eta}_Y^{-1}$, $\varepsilon_Z = \overline{\eta}_Z^{-1}$, $\eta_Y \overline{\varepsilon}_Y=\id_{IP}$ and $\eta_Z \overline{\varepsilon}_Z=\id_{JQ}$ by definition.
The two associated idempotents are $e_Y:= \overline{\varepsilon}_Y\eta_Y $ and $e_Z:= \overline{\varepsilon}_Z\eta_Z $.

Suppose first that there is an equivalence $Y \simeq Z$ of splitting data, \ie an equivalence $F\colon Y\overset{\sim}{\to}Z$ equipped with an isomorphism $\theta\colon FP \overset{\sim}{\Rightarrow} Q$ in~$\cat B$. We must show that $e_Y= e_Z$.
By taking mates, we easily obtain $I \cong JF$ and therefore $IP\cong JQ$. 
But there is an ambiguity: We can take mates with respect either to the left or the right adjunctions, and the resulting isomorphism will either match $\eta_Y$ with $\eta_Z$ or $\overline{\varepsilon}_Y$ with $\overline{\varepsilon}_Z$, accordingly, but \emph{a~priori} not both. 
Luckily, the strong bond between the left and the right adjoint implies the two mates coincide. Let us be more precise.

The isomorphism $JF\cong I$ resulting from $I\dashv P$ amounts to the composite
\[
\phi \colon
\xymatrix{
JF
 \ar@{=>}[r]_-\sim^-{JF \,\overline{\eta}_Y \;} &
JFPI
 \ar@{=>}[r]_-\sim^-{J \theta I} &
JQI
 \ar@{=>}[r]^-{\overline{\varepsilon}_Z I} &
I
}
\]
(this formally requires choosing an adjoint quasi-inverse for~$F$, but then the associated data gets canceled out.)
Using $P\dashv I$ instead, we get the isomorphism 
\[
\psi\colon
\xymatrix{
I 
 \ar@{=>}[r]^-{\eta_Z I} &
JQI 
  \ar@{=>}[r]_-\sim^-{\;\; J \theta^{-1} \!I} &
JFPI 
  \ar@{=>}[r]_-\sim^-{JF \varepsilon_Y\;} &
JF
}
\]
From $\eta_Z \overline{\varepsilon}_Z=\id$ and $\varepsilon_Y = \overline{\eta}_Y^{-1}$, we immediately see that $\psi\phi = \id_{JF}$.
Hence $\phi$ and $\psi$ are mutually inverse. (Incidentally, note that the composite $\phi \psi$ computes as $e_Z I \colon I\Rightarrow I$, showing that $e_Z I = \id_{I}$.)

We thus obtain the isomorphism $(J \theta )(\psi P) \colon IP \overset{\sim}{\Rightarrow} JQ$ with inverse $(\phi P)(J \theta^{-1})$.
To prove that $e_Z = e_Y$, it therefore suffices to show that
\[
(J \theta )(\psi P) \eta_Y = \eta_Z
\quad \textrm{ and } \quad
\overline{\varepsilon}_Y (\phi P)(J \theta^{-1}) = \overline{\varepsilon}_Z .
\]
After expanding the definitions of $\phi$ and~$\psi$, these equations immediately follow from one of the triangular identities for the adjunction $P\dashv I$ or $I \dashv P$, respectively. 

To prove the converse implication, suppose now that $e_Y = e_Z$ for the two given splitting data. We must find an equivalence between the latter.

First of all, since $T:= IP \cong \img(e_Y)$ and $S:= JQ \cong \img(e_Z)$ are two images of the same morphism  in the additive category $\End_\cat B(\Id_X)$, there exists an isomorphism $\theta\colon T\overset{\sim}{\Rightarrow} S$ such that $\theta \eta_Y = \eta_Z$ and $\overline{\varepsilon}_Z \theta = \overline{\varepsilon}_Y$.
In particular, $\theta$ is an isomorphism $(T,\eta_Y) \cong (S,\eta_Z)$ of the localizations of~$X$ induced by the adjunctions $P\dashv I$ and $Q\dashv J$.
By \Cref{Lem:EM-UP}, there exists an equivalence $Y\overset{\sim}{\to}Z$ preserving these adjunctions. 
It is in particular an equivalence of splitting data.
\end{proof}

\begin{Rem} \label{Rem:wbc-vs-ic}
A similar treatment of categorified direct summands can be found in \cite{DouglasReutter18pp}. 
In particular, our uniqueness result corresponds to \cite[Prop.\,1.3.4]{DouglasReutter18pp}; neither result implies the other, strictly speaking, but the ideas are closely related. 
\end{Rem}

%------------------------------------------------------------------------------
\section{Krull--Schmidt bicategories}
\label{sec:KS} %
%------------------------------------------------------------------------------

We are now ready to prove our main results. Indeed, it will be a very easy task because most of the work is already done.

\begin{Def} \label{Def:str-indec}
We say that a nonzero object $X$ of a locally additive bicategory $\cat B$ is \emph{strongly indecomposable} if the commutative ring $\End_{\cat B(X,X)}(\Id_X)$ is connected.
\end{Def}

Let us justify the use of ``strongly'':

\begin{Rem} \label{Rem:str-indec-vs-indec} 
Strongly indecomposable objects are indecomposable. 
Indeed, nontrivial decompositions $X\simeq X_1 \oplus X_2$ produce (as in \Cref{Rem:matrix-not}) nontrivial decompositions $1_X = e_1 + e_2$  in orthogonal idempotents of $\End_\cat B(\Id_X)$, where $e_i$ is the idempotent associated to the direct summand~$X_i$ as in \Cref{Def:assoc-idemp}.
\end{Rem}

\begin{Thm} 
\label{Thm:KRS}
Let $X$ be an object of an additive bicategory~$\cat B$ which admits 
 two direct sum decompositions
\[
X\simeq Y_1 \oplus \cdots \oplus Y_n
\quad \textrm{ and }\quad
X \simeq Z_1 \oplus \cdots \oplus Z_m
\]
such that all $Y_k$ and $Z_\ell$ are strongly indecomposable (\Cref{Def:str-indec}).
Then $n=m$ and, after renumbering the factors, there are equivalences $Y_k \simeq Z_k$ for all~$k$.
\end{Thm}

\begin{proof}
By construction, the identities $1_{Y_k}$ ($k=1,\ldots,n$) correspond under the given equivalences to a complete family of orthogonal primitive idempotents $e_k$ in the commutative ring $\End_\cat B(\Id_X)$.
Here each $e_k$ is the idempotent associated with the direct summand~$Y_k$.
Similarly, the direct summands $Z_\ell$ yield a complete family of orthogonal primitive idempotents~$f_\ell$.
But then clearly, this being a commutative ring, each of the $e_k$ must be equal to exactly one of the~$f_\ell$.
Thus $n=m$ and we find a permutation $\sigma$ such that $e_k = f_{\sigma(k)}$ for all~$k$.
It follows by \Cref{Prop:idemp-char} that $Y_k$ and $Z_{\sigma(k)}$ are equivalent, even as direct summands (\ie splitting data) of~$X$.
\end{proof}

\begin{Def} \label{Def:KS-bicat}
We say an additive bicategory is \emph{Krull--Schmidt} if every nonzero object admits a direct sum decomposition into strongly indecomposable objects.
\end{Def}

\begin{Rem}
In a weakly block complete bicategory  (\Cref{Def:bc}), indecomposable objects and strongly indecomposable objects coincide. Indeed, by definition, each nontrivial decomposition of the 2-cell $1_X$ as a sum of orthogonal idempotents arises from a nontrivial direct sum decomposition of the object~$X$.
\end{Rem}

\begin{Cor} \label{Cor:KRS}
In a Krull--Schmidt bicategory (\Cref{Def:KS-bicat}), indecomposable and strongly indecomposable objects coincide, and every object decomposes in an essentially unique way into a direct sum of indecomposable objects.
\end{Cor}

\begin{proof}
Strongly indecomposable objects are indecomposable by \Cref{Rem:str-indec-vs-indec}.
By hypothesis every object $X$ admits a decomposition $X\simeq X_1\oplus\cdots \oplus X_n$ into strongly indecomposables. If $X$ is indecomposable, $n=1$ hence $X\simeq X_1$ is strongly indecomposable.
The remaining statement now follows from \Cref{Thm:KRS}. 
\end{proof}

\begin{Def} \label{Def:semi-con}
As in the introduction, we  say a commutative ring is \emph{semiconnected} if it is a direct product of finitely many connected rings.
\end{Def}

\begin{Thm} \label{Thm:KS-char}
An additive bicategory is Krull--Schmidt if and only if it is weakly block complete and every 2-cell endomorphism ring is semiconnected.
\end{Thm}

\begin{proof}
Consider an idempotent 2-cell $e\colon \Id_X\Rightarrow \Id_X$ on an object~$X$ of  a Krull--Schmidt bicategory~$\cat B$. 
Let $X\simeq X_1 \oplus \cdots \oplus X_n$ be a decomposition in strongly indecomposable objects~$X_k$. 
This equivalence induces a decomposition $1_X = e_1 +\cdots + e_n$ in orthogonal primitive idempotents, where $e_k\in \End_\cat B(\Id_X)$ is the idempotent associated to~$X_k$.
It follows that $e= \sum_{k\in J} e_k$ for some subset $J\subseteq \{1,\ldots,n\}$.
By \Cref{Prop:idemp-char}, $X \simeq \big( \bigoplus_{k\in J} X_k \big) \oplus \big( \bigoplus_{k\not\in J} X_k \big)$ 
realizes  the decomposition $1_X = e + (1_X - e)$ at the level of objects. 
This shows that $\cat B$ is weakly block complete. 
Also, each ring $\End_\cat B(\Id_{X_k})\cong e_k \End_\cat B(\Id_X)$ is connected as $e_k$ is primitive.

Now suppose that $\cat B$ is weakly block complete and that for each object $X$ there is a decomposition $\End_\cat B(\Id_X)\cong R_1 \times \cdots \times R_n$ in finitely many connected rings~$R_k$. 
Then $1_X$ is a sum of $n$ primitive idempotents and, by weak block completeness, there follows a decomposition $X\simeq X_1 \oplus \cdots \oplus X_n$ with $\End(\Id_{X_k})\cong R_k$. 
In particular, each $X_k$ is strongly indecomposable. 
This shows that $\cat B$ is Krull--Schmidt.
\end{proof}

%------------------------------------------------------------------------------
\section{Notions of strong indecomposability}
\label{sec:indecs} %
%------------------------------------------------------------------------------

In this section we compare various candidates for the notion of strong indecomposable objects in an additive bicategory (\Cref{Cor:strongly-indec-X}). 
This further justifies our choice in \Cref{Def:str-indec}, if needed, and relates it to its 1-categorical analog.

\begin{Lem} \label{Lem:comm-matrices}
Suppose $A$ and $B$ are two objects in an additive category such that the endomorphism ring $\End(A\oplus B)$ of their sum is commutative. 
Then $A$ and $B$ are Hom-orthogonal and therefore $\End(A\oplus B)=\End(A)\times \End(B)$.
\end{Lem}

\begin{proof}
For $x\in \Hom(A,B)$ and $y\in \Hom(B,A)$, consider the elements $u=[\,{}^1_0 \,{}^0_0\,]$, $v=[\,{}^0_0\,{}^0_1\,]$, $f= [\,{}^0_x\,{}^0_0\,]$ and $g=[\,{}^0_0\,{}^y_0\,]$ of 
$\End(A\oplus B)$. By commutativity, the calculations $vf = f$ and $fv=0$ imply $x=0$, whereas $ug = g$ and $gu=0$ imply $y=0$.
\end{proof}

\begin{Prop} \label{Prop:indecs-general}
Consider the following properties of a (nonzero) additive mon\-oid\-al category~$\cat A$, \ie a category which is additive and monoidal and whose tensor functor $\otimes \colon \cat A\times \cat A\to \cat A$ is additive in both variables:
\begin{enumerate} [\rm(1)]
\item \label{it:localunit}
The commutative endomorphism ring $\End_\cat A(\Unit)$ of the tensor unit $\Unit$ is connected.
\item \label{it:unit}
The tensor unit object $\Unit$ is indecomposable in~$\cat A$ for the direct sum.
\item \label{it:local}
If a sum of two objects of $\cat A$ is tensor-invertible, one of them must be too.
\item \label{it:tenscat}
If $\cat A$ is equivalent as a monoidal category to a product $ \cat A_1 \times \cat A_2$ of two additive monoidal categories, then $\cat A_1\simeq 0$ or $\cat A_2\simeq 0$.
\end{enumerate}
Then   \eqref{it:localunit}$\Rightarrow$\eqref{it:unit}$\Rightarrow$\eqref{it:local}$\Rightarrow$\eqref{it:tenscat}, and if $\cat A $ is idempotent complete all four are equivalent.
\end{Prop}

\begin{proof}
\eqref{it:localunit}$\Rightarrow$\eqref{it:unit}:
Suppose that $\Unit \cong E_1 \oplus E_2$. 
Since $\End(\Unit)$ is commutative, we have $\End(\Unit)\cong \End(E_1) \times \End(E_2)$ by  \Cref{Lem:comm-matrices}. 
Assuming $\End(\Unit)$ is connected, we must therefore have $\End(E_i)=0$ and thus $E_i\cong 0$ for $i=1$ or $i=2$.

\eqref{it:unit}$\Rightarrow$\eqref{it:local}:
Suppose $A\oplus B$ has a two-sided tensor inverse~$C$, with $A,B,C\in \cat A$. Then $\Unit \cong (A \oplus B)\otimes C\cong (A \otimes C)\oplus (B \otimes C)$, hence if $\Unit$ is indecomposable we must have $\Unit \cong A \otimes C$ or $\Unit \cong B \otimes C$. Since $C$ is invertible, either $A$ or $B$ must be too.

\eqref{it:local}$\Rightarrow$\eqref{it:tenscat}:
Given an equivalence $\cat A\simeq \cat A_1 \times \cat A_2$ of additive monoidal categories, there follows in $\cat A$ a decomposition $\Unit \cong E_1 \oplus E_2$, where $E_i$ corresponds to the tensor unit of~$\cat A_i$. 
Since $\Unit$ is tensor invertible in~$\cat A$, then by~\eqref{it:local} either $E_1$ or~$E_2$ must be as well. 
Say $E_1$ is, with inverse~$F$. 
Since $E_1\otimes \cat A_2 \simeq 0$, it follows that $\cat A_2 = \Unit \otimes \cat A_2 = F \otimes E_1 \otimes \cat A_2 \simeq 0$.
Thus \eqref{it:tenscat} holds.

Finally, we show \eqref{it:tenscat}$\Rightarrow$\eqref{it:localunit} assuming that $\cat A$ is idempotent complete.
Let $\id_\Unit = e_1 + e_2$ for two orthogonal idempotents $e_1,e_2\colon \Unit \to \Unit$.
For every object $A \in \cat A$, write $e_i\!\!\downharpoonright_A$ for the endomorphism of $A$ corresponding to $e_i \otimes A$ under the structural isomorphism $A \cong \Unit \otimes A$.
Note that $e_1\!\!\downharpoonright_A$ and $e_2\!\!\downharpoonright_A$ are orthogonal idempotents on~$A$ with $\id_A = e_1 \!\!\downharpoonright_A + \; e_2 \!\!\downharpoonright_A$. 
By idempotent completion, we get a decomposition $A \cong A_1 \oplus A_2$ where $A_i = \img(e_i\!\!\downharpoonright_A)$.
One easily checks that this direct sum decomposition is functorial in~$A$ and gives rise to a decomposition $\cat A \simeq \cat A_1 \times \cat A_2$ of $\cat A$ as an additive monoidal category, where $\cat A_i$ is the full subcategory $\{A \in \cat A\mid e_i \!\!\downharpoonright_A = \id_A\}$ of~$ \cat A$.
By~\eqref{it:tenscat}, we must have for either~$i$ that $\cat A_i\simeq 0$, that is that $\img(e_i\!\!\downharpoonright_A)\cong 0$ for all $A$ and thus $e_i=0$.
This shows that the ring $\End_\cat A(\Unit)$ is connected.
\end{proof}

\begin{Rem} \label{Rem:non-tensor-variants}
In the situation of \Cref{Prop:indecs-general}, we could also consider the following two properties of the category $\cat A$ which do not depend on its monoidal structure:  
\begin{enumerate} [\rm(1)$'$]
\item \label{it:Id}
The commutative ring $\End(\Id_\cat A)$ of natural transformations $\Id_\cat A\Rightarrow \Id_\cat A$, \ie the categorical ``center'' of~$\cat A$, is connected.
\setcounter{enumi}{3}
\item \label{it:cat}
If $\cat A\simeq \cat A_1 \times \cat A_2$ as an additive category then $\cat A_1\simeq 0$ or $\cat A_2\simeq 0$. 
\end{enumerate}
One can verify that \eqref{it:Id}$'$$\Rightarrow$\eqref{it:localunit}, \eqref{it:cat}$'$$\Rightarrow$\eqref{it:tenscat} and \eqref{it:Id}$'$$\Rightarrow$\eqref{it:cat}$'$ in general, and that \eqref{it:Id}$'$$\Leftrightarrow$\eqref{it:cat}$'$ if $\cat A$ is idempotent complete. 
But \eqref{it:Id}$'$ and \eqref{it:cat}$'$ do not seem  otherwise relevant here.
\end{Rem}

The endomorphism category of any object in a locally additive bicategory can be viewed as an additive monoidal category, with tensor product given by horizontal composition. In this case, \Cref{Prop:indecs-general} translates as follows:

\begin{Cor} 
\label{Cor:strongly-indec-X}
For any (nonzero) object $X$ of a locally additive bicategory~$\cat B$, we have implications \eqref{it:localunit-X}$\Rightarrow$\eqref{it:unit-X}$\Rightarrow$\eqref{it:local-X}$\Rightarrow$\eqref{it:tenscat-X} between the following properties, and all four are in fact equivalent if the additive category $\End_\cat B(X)$ is idempotent complete:
\begin{enumerate} [\rm(1)]
\item \label{it:localunit-X}
The commutative endomorphism ring $\End_\cat B(\Id_X)$ is connected.
\item \label{it:unit-X}
The identity 1-cell $\Id_X$ is indecomposable in the additive category~$\End_\cat B(X)$.
\item \label{it:local-X}
If a sum of two 1-cells $X\to X$ is an equivalence, then one of them must be too.
\item \label{it:tenscat-X}
$\End_\cat B(X)$ is indecomposable as an additive monoidal category.
\qed
\end{enumerate}
\end{Cor}

\begin{Rem} \label{Rem:str-indec-1cat}
Property \eqref{it:local-X} is the evident 2-categorical analog of being a \emph{strongly indecomposable} object $A$ of an additive (1-)category~$\cat A$. Recall the latter means that $\End_\cat A(A)$ is a local ring, that is: If a sum $f+g$ of two morphisms $f,g\colon A\to A$ is invertible then either $f$ or $g$ is invertible. 
Since we do not wish to assume idempotent completion of our Hom categories, we are left with the choice of at least four reasonable 2-categorical analogs of strong indecomposability.
In \Cref{Def:str-indec}, we chose the strongest and most useful one. 
It is also the only one which always implies the usual indecomposability and which is stable under (weak) block completion.
\end{Rem}

%------------------------------------------------------------------------------
\section{Direct comparison of direct summands}
\label{sec:bc} %
%------------------------------------------------------------------------------

In this section we explain how to categorify the usual proof of the uniqueness of decompositions in strongly indecomposables.
This does not use the characterization of direct summands via idempotents (\Cref{Prop:idemp-char}), but instead requires us to understand how to ``cut down'' a direct summand by another (\Cref{Prop:crucial-splitting}). 
It is a considerably harder task than its 1-categorical analog, which is a triviality.

\begin{Lem} \label{Lem:splitting}
Suppose that $\cat B$ is weakly block complete, and that a 1-cell $E\colon X\to X$ is part of an idempotent monad $(E,\mu, \iota)$ and of an idempotent comonad $(E,\delta, \epsilon)$ such that $\iota\epsilon = \id_E$.
Then there is a decomposition $X\simeq Y\oplus Z$ such that the direct summand $Y$ induces the given monad and comonad as in \Cref{Rem:monad}.
\end{Lem}

\begin{proof}
Since $\iota\epsilon = \id_E$, the 2-cell $e:= \epsilon \iota\colon  \Id_X \Rightarrow E \Rightarrow \Id_X$ is an idempotent element of the ring  $\End_\cat B(\Id_X)$, and so is $f:= 1_X - e$.
By weak block completion, we deduce a direct sum decomposition $X \simeq Y \oplus Z$ such that $\Id_Y$ and $\Id_Z$ correspond to $\img(e)$ and $\img(f)$ respectively. 
Moreover, since $\Id_X \cong \img(e)\oplus \img(f)$ and since $\epsilon$ is a split mono by hypothesis, we have $\img(e)= \img (\iota) \cong E$ in the category $\End_\cat B(X)$, compatibly with the respective structure 2-cells. We can now easily conclude.
\end{proof}

\begin{Prop} 
\label{Prop:crucial-splitting}
Let $Y$ and $Z$ be two direct summands (for two different decompositions) of an object~$X$ in a weakly block complete additive bicategory~$\cat B$, with structural 1-cells denoted respectively by
\[
Y \overset{I}{\longrightarrow} X \overset{P}{\longrightarrow} Y
\quad \textrm{ and } \quad
Z \overset{J}{\longrightarrow} X \overset{Q}{\longrightarrow} Z .
\]
Suppose that $PJQI$ is an equivalence $Y\overset{\sim}{\to} Y$.
Then $QI\colon Y\to Z$ and $PJ\colon Z\to Y$ are part of a direct sum decomposition $Z \simeq Y\oplus Y'$.
\end{Prop}

\begin{proof}
By \Cref{Lem:splitting}, it suffices to show that $E:= (QI)(PJ)$ is both part of an idempotent monad $(E, \mu, \iota)$ and an idempotent comonad $(E,\delta, \epsilon)$ such that  $\iota \epsilon = \id_{E}$.
Suppose $P,I$ and $Q,J$ are part of two splitting data for~$X$ as usual, and write the units and counits of the four adjunctions $I\dashv P\dashv I$ and $J \dashv Q \dashv J$ as follows:
\begin{align*}
\eta_Y \colon \Id_X\Rightarrow IP 
\quad \quad
\varepsilon_Y \colon PI \overset{\sim}{\Rightarrow} \Id_Y
\quad \quad
\overline{\eta}_Y \colon \Id_Y \overset{\sim}{\Rightarrow} PI
\quad\quad
\overline{\varepsilon}_Y \colon IP \Rightarrow \Id_X \\
\eta_Z \colon \Id_X\Rightarrow JQ 
\quad \quad
\varepsilon_Z \colon QJ \overset{\sim}{\Rightarrow} \Id_Z
\quad \quad
\overline{\eta}_Z \colon \Id_Z \overset{\sim}{\Rightarrow} QJ
\quad\quad
\overline{\varepsilon}_Z \colon JQ \Rightarrow \Id_X
\end{align*}
Thus we have $\varepsilon_Y = (\overline{\eta}_Y)^{-1}$ and $\eta_Y \overline{\varepsilon}_Y = \id_{IP}$, and similarly for~$Z$.
Then the two composite adjunctions $PJ\dashv QI \dashv PJ$ have the following units and counits:
\begin{align*}
\eta \colon \xymatrix{ \Id_Z \ar@{=>}[r]^-{\overline{\eta}_Z}_-\sim & QJ \ar@{=>}[r]^-{Q \, \eta_Y J } & QIPJ }
\quad\quad\quad\quad
\varepsilon \colon \xymatrix{ PJQI \ar@{=>}[r]^-{P\, \overline{\varepsilon}_ZI} & PI \ar@{=>}[r]^-{\varepsilon_Y}_-\sim & \Id_Y }
\\
\overline{\eta} \colon  \xymatrix{ \Id_Y \ar@{=>}[r]^-{\overline{\eta}_Y}_-\sim & PI \ar@{=>}[r]^-{P \, \eta_Z I} & PJQI }
\quad\quad\quad\quad
\overline{\varepsilon} \colon \xymatrix{ QI PJ \ar@{=>}[r]^-{Q \, \overline{\varepsilon}_Y J } & QJ \ar@{=>}[r]^-{\varepsilon_Z}_-\sim & \Id_Z  }
\end{align*}
The monad $(E, \mu, \iota)$ and the comonad $(E,\delta, \epsilon)$ on~$Z$ are those induced by $PJ\dashv QI$ and $QI \dashv PJ$, respectively. In particular $\iota = \eta$ and $ \epsilon = \overline{\varepsilon}$, so that the required relation
 $ \iota \epsilon = \eta \overline{\varepsilon} = \id_{QIPJ} $
immediately follows from $\varepsilon_Z = (\overline{\eta}_Z)^{-1}$ and $\eta_Y \overline{\varepsilon}_Y = \id_{IP}$.

It remains to verify that this monad and this comonad are idempotent.
For the monad, recall that by definition its multiplication is $\mu = (QI) \varepsilon (PJ)$, that is
\[
\mu \colon  
\xymatrix{ 
{ E  E = QI PJ QI PJ}
  \ar@{=>}[rr]^-{QI P \, \overline{\varepsilon}_Z  IPJ} &&
QI PI PJ 
  \ar@{=>}[rr]^-{QI  \varepsilon_Y PJ}_-\sim && QIPJ = E
 },
\]
hence we must show that $QI P\, \overline{\varepsilon}_Z  IPJ$ is invertible. 

In fact, we can show that $P\,\overline{\varepsilon}_Z I$ is already invertible.
Indeed, the composite $PJQI$ is an equivalence by hypothesis. 
It follows that the composite adjunction $PJQI \dashv PJQI$ is an adjoint equivalence. 
In particular, its counit 
\[
\alpha\colon
\xymatrix@C=16pt{
PJQIPJQI 
  \ar@{=>}[rr]^-{\;\, PJQ \,\overline{\varepsilon}_Y JQI} &&
PJQJQI 
  \ar@{=>}[rr]^-{PJ  \varepsilon_Z QI}_-\sim &&
PJQI
  \ar@{=>}[rr]^-{P\, \overline{\varepsilon}_Z I} &&
PI 
  \ar@{=>}[rr]^-{\varepsilon_Y}_-\sim &&
\Id_Y 
}
\]
is an isomorphism.
Thus $P \,\overline{\varepsilon}_Z I$ is a split epi, since $(PJ \varepsilon_Z QI)(PJQ \,\overline{\varepsilon}_Y JQI)\alpha^{-1}\varepsilon_Y$ provides a section for it.
But $P \,\overline{\varepsilon}_Z I$ is also a split mono, because $\overline{\varepsilon}_Z$ is a split mono by construction.
Therefore $P \,\overline{\varepsilon}_Z I$ is invertible as claimed.

The proof that the comonad is idempotent is similar and is left to the reader.
\end{proof}

\begin{proof}[Alternative proof of \Cref{Thm:KRS}]
We start by replacing $\cat B$ with its block completion $\cat B^\flat$ as in \Cref{Thm:bc}. 
Note that the embedding $\cat B\hookrightarrow \cat B^\flat$ preserves direct sums (because it is an additive pseudofunctor) and strongly indecomposable objects (because it is 2-fully faithful). 
Also by being 2-fully faithful, it reflects whether two given 1-cells of~$\cat B$ form (adjoint) equivalences, as we are setting out to establish.
Thus we may and will assume that $\cat B$ is block complete. 

For all $k$ and~$\ell$, let
\[
Y_k \overset{I_k}{\longrightarrow} X \overset{P_k}{\longrightarrow} Y_k 
\quad \textrm{ and } \quad
Z_\ell \overset{J_\ell}{\longrightarrow} X \overset{Q_\ell}{\longrightarrow} Z_\ell 
\]
denote the inclusion and projection 1-cells of the two given direct sum decompositions of~$X$, and let $E_k := I_k P_k$ and $F_\ell := J_\ell Q_\ell$ be the corresponding localizations of~$X$. The isomorphism
\[
\Id_X \cong F_1 \oplus \ldots \oplus F_s
\]
in $\End_\cat B(X)$ induces an isomorphism
\[
\Id_{Y_1} 
\cong P_1 I_1 
\cong P_1(F_1 \oplus \ldots \oplus F_s)I_1
\cong P_1F_1I_1 \oplus P_1F_2I_1 \oplus \ldots \oplus P_1F_sI_1
\]
in $\End_\cat B(Y_1)$. Since $Y_1$ is strongly indecomposable, by \Cref{Cor:strongly-indec-X}\,\eqref{it:local-X} we can find an $\ell\in \{1,\ldots , s\}$ such that $P_1F_\ell I_1$ is an equivalence $Y_1 \overset{\sim}{\to} Y_1$. 
By renumbering the~$Z_\ell$'s, we may assume that $\ell =1$. 
Thus
\begin{equation} \label{eq:PFI}
P_1 F_1 I_1 \cong (P_1 J_1) (Q_1 I_1) 
\end{equation}
is an equivalence. 
By \Cref{Prop:crucial-splitting} and the (weak) block completeness of~$\cat B$, the two composites $Q_1I_1\colon Y_1\to Z_1$ and $P_1J_1\colon Z_1\to Y_1$ are part of a direct sum decomposition $Z_1 \simeq Y_1\oplus Y_1'$.
As $Z_1$ is strongly indecomposable it is also indecomposable (\Cref{Rem:str-indec-vs-indec}), hence we must have $Y_1'\simeq 0$, that is $Q_1I_1$ and $P_1J_1$ are in fact an adjoint equivalence $Y_1 \simeq Z_1$ of direct summands of~$X$. 

Note that, by construction, the latter equivalence is the component $Y_1\to Z_1$ of the given composite equivalence (whose components are precisely the $Q_\ell I_k$)
\begin{equation} \label{eq:or-comp-eq}
Y_1 \oplus \underbrace{Y_2 \oplus \ldots \oplus Y_n}_{=:\, Y'} 
\overset{\sim}{\to} 
X 
\overset{\sim}{\to} 
Z_1 \oplus \underbrace{ Z_2 \oplus \ldots \oplus Z_m}_{=: \,Z'} \,.
\end{equation}
We claim that \eqref{eq:or-comp-eq} similarly restricts to an equivalence $Y'\simeq Z'$,
after which the conclusion of the theorem will follow by a routine inductive argument.
To prove the claim, note that the equivalence $Y_1\overset{\sim}{\to} Z_1$ identifies (up to isomorphism) $I_1$ with $J_1$ and $P_1$ with~$Q_1$. 
In particular we have $Q_\ell I_1\cong 0$ and $P_1 J_\ell\cong 0$ for $\ell \neq 1$, which means that \eqref{eq:or-comp-eq} has a diagonal $2\times2$ matrix form with respect to the sum decompositions $Y_1 \oplus Y'$ and $Z_1 \oplus Z'$. 
As \eqref{eq:or-comp-eq} is an equivalence, it follows that the second diagonal component $Y'\to Z'$ is also an equivalence, as claimed.
\end{proof}

%------------------------------------------------------------------------------
\section{Examples}
\label{sec:exas} %
%------------------------------------------------------------------------------

We finally consider some examples of Krull--Schmidt bicategories, concentrating on various kinds of ``2-representation theory''.
We begin with a few generalities:

\begin{Prop} \label{Prop:wb-crit}
Suppose $\cat B$ is an additive bicategory where the 2-cell endomorphism ring $\End_\cat B(\Id_X)$ of every object~$X$ is semiconnected (\Cref{Def:semi-con}). 
Then both its block completion~$\cat B^\flat$ (\Cref{Thm:bc}) and its weak block completion~$\cat B^{w\flat}$ (\Cref{Rem:wbc}) are Krull--Schmidt bicategories.
\end{Prop}

\begin{proof}
The block completion and weak block completion are both weakly block complete bicategories. 
Moreover, the canonical embeddings $\cat B \hookrightarrow \cat B^\flat$ and $\cat B\hookrightarrow \cat B^{w\flat}$ are 2-fully faithful pseudofunctors, and by construction the 2-cell endomorphism ring of every object in either completion is a direct factor of that of some object of~$\cat B$ (\cf \Cref{Rem:matrix-not}). Since direct factors of semiconnected rings are semiconnected, we conclude with the characterization  of \Cref{Thm:bc}.
\end{proof}

Clearly not every commutative ring is semiconnected (consider $R= \prod_{\mathbb N} \mathbb Z$), hence we need some criteria. 
The following two lemmas already go a long way:

\begin{Lem} \label{Lem:semi-conn-noetherian}
Every noetherian commutative ring is semiconnected.
\end{Lem}

\begin{proof}
This can be deduced from the fact that the spectrum of a noetherian ring is a finite union of irreducible closed subsets (\cite[Prop.\,I.1.5 and Ex.\,II.2.13]{Hartshorne77}): The latter being connected, there are only finitely many connected components.
\end{proof}

\begin{Lem} \label{Lem:sc-subrings}
Unital subrings of semiconnected rings are semiconnected.
\end{Lem}

\begin{proof}
This boils down to injective morphisms of commutative rings inducing maps with dense image between Zariski spectra (\cite[Ch.\,1 Ex.\,21(v)]{AtiyahMacdonald69}): In particular, they induce surjective maps between sets of connected components.
\end{proof}

The following criterion follows immediately from \Cref{Lem:semi-conn-noetherian} and~\Cref{Thm:KS-char}.

\begin{Cor} \label{Cor:KSbicat-homfin}
Suppose $\cat B$ is a weakly block complete bicategory which is $\kk$-linear over some noetherian commutative ring $\kk$ and such that every 2-cell Hom space is finitely generated as a $\kk$-module. 
Then $\cat B$ is a Krull--Schmidt bicategory. 
\qed
\end{Cor}

We now turn to more concrete situations, beginning with a non-example:

\begin{Exa} 
Recall from \Cref{Exa:ADD} the additive 2-category $\ADD$ of additive categories.
Consider its 2- and 1-full additive sub-bicategory $\ADDic$ of idempotent complete categories:
It is block complete (see \cite[A.7.18]{BalmerDellAmbrogio20}) but not Krull--Schmidt, because it contains $\cat A=\Mod R$ for rings $R$ whose center $\mathrm Z(R)\cong \End(\Id_\cat A)$ is not semiconnected. Here, and again below, we use the well-known identification of the center of a ring with the center $\End(\Id_{\Mod R})$ of its module category.
\end{Exa}

\begin{Exa} \label{Exa:add}
Consider the full 2-subcategory $\cat B$ of $\ADDic$ whose objects are the idempotent complete additive categories $\cat A$ such that $\End(\Id_\cat A)$ is semiconnected. 
Then $\cat B$ is Krull--Schmidt. 
Indeed, if $1_\cat A= e_1 + e_2 \in \End(\Id_{\cat A})$ for such a category $\cat A$ with $e_1,e_2$ two orthogonal idempotents, we get a corresponding decomposition $\cat A \simeq \cat A_1 \oplus \cat A_2$ in $\ADDic$, as the latter is weakly block complete. 
But then $\End(\Id_\cat A) = R_1 \times R_2$ with $R_k \cong \End(\Id_{\cat A_k})$, hence the categories $\cat A_k$ are again in $\cat B$ since direct factors of semiconnected rings are semiconnected.
Thus $\cat B$ is weakly block complete and we conclude with \Cref{Thm:KS-char}.
\end{Exa}

\begin{Exa} \label{Exa:rings}
Many variations of the bicategory $\cat B$ in \Cref{Exa:add}  are possible. 
For instance, we may consider the full subbicategory $\cat B'\subset \cat B$ of module categories $\cat A= \Mod R$ over rings $R$ whose center is semiconnected. 
(Indeed, if $\Mod R\cong \cat A_1 \oplus \cat A_2$ in $\ADDic$ for idempotents $e_1,e_2\in \End(\Id_{\Mod R})\cong \mathrm Z(R)$, we have $\cat A_k \simeq \Mod R_k$ for the rings $R_k:=e_k R e_k = Re_k$ ($k=1,2$). Moreover, each $\mathrm Z(R_k)$ is semiconnected since $\mathrm Z(R)$ is and $\mathrm Z(R) \cong \mathrm Z(R_1)\times \mathrm Z(R_2)$.)
Similarly, we may further restrict to rings~$R$ with noetherian center (by \Cref{Lem:semi-conn-noetherian}) and/or replace $\Mod R$ with the category of finitely presented/generated modules, or projective modules, finitely generated projectives, etc., indeed any option for which the previous splitting argument still goes through. We may also take module categories over finite dimensional algebras $R$  over a fixed field~$\kk$ and $\kk$-linear functors between them, since the center of a finite dimensional algebra is again finite dimensional hence noetherian.
\end{Exa}

\begin{Exa} \label{Exa:bimod}
Consider the bicategory of rings, bimodules and bimodules maps, and its sub-bicategory $\cat B''$ of rings whose center is semiconnected. 
Then $\cat B''$ is Krull--Schmidt.
To see this, just recall that by the Eilenberg-Watts theorem we may identify $\cat B''$ with the 2-full sub-bicategory of $\cat B' \subset \ADDic$ as in \Cref{Exa:rings} having the same objects but where the only 1-cells are the colimit-preserving functors.
Since all functors involved in direct sum decompositions are colimit-preserving (they all have two-sided adjoints), the same object decompositions of $\cat B'$ equally work in~$\cat B''$, hence $\cat B''$ is also weakly block complete.
As before, we may replace ``semiconnected'' with ``noetherian'', or we may work over a fixed~$\kk$.
\end{Exa}

Our original motivation for this article is the following application:

\begin{Cor} \label{Exa:M2M}
Let $\kk$ be any noetherian commutative ring.
Every bicategory of $\kk$-linear \emph{Mackey 2-motives} in the sense of \cite{BalmerDellAmbrogio20} is Krull--Schmidt.
Similarly, every bicategory of \emph{cohomological} Mackey 2-motives as in \cite{BalmerDellAmbrogio21pp} is Krull--Schmidt.
By \Cref{Thm:KRS}, we deduce the uniqueness of the motivic decomposition of every finite group in each of these bicategories.
\end{Cor}

\begin{proof}
These bicategories are constructed as block completions, hence in particular they are weakly block complete. 
Moreover their 2-cell Hom spaces are easily seen to be finitely generated free $\kk$-modules. 
We may thus apply \Cref{Prop:wb-crit} or  \Cref{Cor:KSbicat-homfin}. (The cohomological case is also a variant of \Cref{Exa:bimod}.)
\end{proof}
\begin{center} $*\;*\;*$ \end{center}

From now on we look at various flavors of linear 2-representation theory, starting with some commonly used target 2-categories for said 2-representations:

\begin{Exa} \label{Exa:2vect}
Let $\kk$ be any field. Consider the 2-category $\mathsf{2FVect}_\kk$ of \emph{finite dimensional 2-vector spaces} (over~$\kk$) in the sense of Kapranov--Voevodsky \cite{KapranovVoevodsky94} and Neuchl~\cite{NeuchlPhD}.
Recall that $\mathsf{2FVect}_\kk$ can be defined as the $\kk$-linear 2-category of finite semisimple $\kk$-linear categories, $\kk$-linear functors and natural transformations.
Here a $\kk$-linear category is \emph{finite semisimple} if it is abelian and there is a finite set of absolutely simple objects (objects with one-dimensional endomorphism algebra) such that every other object is a direct sum of these in an essentially unique way. (Each finite semisimple category is in fact equivalent to $\mathrm{vect}_\kk^n$ for some~$n$.)
One easily checks that $\mathsf{2FVect}_\kk$ is block complete and all its 2-cell Hom $\kk$-vector spaces are finite dimensional. In particular it is Krull--Schmidt by \Cref{Cor:KSbicat-homfin}.
\end{Exa}

\begin{Exa} \label{Exa:2hilb}
Consider the 2-category $\textsf{2FHilb}$ of \emph{finite-dimensional 2-Hilbert spaces} in the sense of Baez~\cite{Baez97} (see also \cite{HeunenVicary19}). 
Recall that a 2-Hilbert space is a dagger category which is enriched on finite dimensional Hilbert spaces (compatibly with adjoints and inner products) and has finite direct sums and split idempotents; it is finite dimensional if it admits a finite set of simple objects generating all objects under finite direct sums.
Then $\textsf{2FHilb}$ is Krull--Schmidt. 
Indeed, there is a 2-functor $\textsf{2FHilb}\overset{\sim}{\to} \textsf{2FVect}_\mathbb C$ forgetting daggers and inner products which is a biequivalence of bicategories (this is analogous to the equivalence between the 1-categories of finite dimensional Hilbert spaces and finite dimensional complex vector spaces).
\end{Exa}

\begin{Exa} \label{Exa:CATfin}
Let $\kk$ be a field. 
A $\kk$-linear category is \emph{finite} if it is equivalent to that of finite dimensional representations of some finite dimensional associative unital $\kk$-algebra.
The 2-category $\CAT_\kk^{\mathsf{fin}}$ of finite $\kk$-linear categories, $\kk$-linear functors and natural transformations is Krull--Schmidt; it is indeed a variation of \Cref{Exa:rings}.
\end{Exa}

\begin{Exa} \label{Exa:CATtrfin}
For a field~$\kk$, consider the 2-category $\mathsf{TRI}_\kk^{\mathsf{fin}}$ whose objects are triangulated categories equivalent to the bounded derived category of some finite $\kk$-linear category (\Cref{Exa:CATfin}), triangulated functors and triangulated natural transformations. We leave it as an exercise to verify that $\mathsf{TRI}_\kk^{\mathsf{fin}}$ is Krull--Schmidt.
\end{Exa}

The next result will essentially take care of all our remaining examples:

\begin{Thm} \label{Thm:KS-PsFun}
Let $\kk$ be any commutative ring (for instance $\kk = \mathbb Z$ or a field).
Consider the $\kk$-linear bicategory $\PsFun_\kk(\cat B, \cat C)$ of $\kk$-linear pseudofunctors between two  $\kk$-linear bicategories $\cat B$ and~$\cat C$ (\Cref{Ter:PsFun}). 
Then $\PsFun_\kk(\cat B, \cat C)$ is Krull--Schmidt if $\cat C$ is Krull--Schmidt and $\cat B$ has finitely many equivalence classes of objects.
\end{Thm}

\begin{proof}
By replacing $\cat B$ with a biequivalent sub-bicategory if necessary, we may assume it has finitely many objects.
If $\cat C$ is weakly block complete then so  is $\PsFun_\kk(\cat B, \cat C)$ by \Cref{Prop:levelwise-wbc}.
Now recall that for (non-necessarily $\kk$-linear) pseudofunctors $\cat F_1, \cat F_2\colon \cat B\to \cat C$ and transformations $t,s\colon \cat F_1\Rightarrow \cat F_2$, a modification $M\colon t\Rrightarrow s$ consists of a family $\{M_X\}_X$ of 2-cells $M_X\colon t_X\Rightarrow s_X$ of~$\cat C$ indexed by the objects $X$ of~$\cat B$, satisfying some relations.
In particular, for every $\kk$-linear pseudofunctor $\cat F\colon \cat B\to \cat C$ the ring $\End(\Id_\cat F)$ of self-modifications of the identity transformation is a unital subring of the product $\prod_{X\in \Obj \cat B} \End_{\cat C(\cat FX, \cat FX)}(\Id_{\cat FX})$.
Being a product of finitely many semiconnected rings, the latter is semiconnected. 
Hence so is $\End(\Id_\cat F)$ by \Cref{Lem:sc-subrings}.
Once again, we conclude with \Cref{Thm:KS-char}.
\end{proof}

\begin{Rem} \label{Rem:KS-2Fun}
If $\cat B$ and $\cat C$ are 2-categories with $\cat C$ Krull--Schmidt, the same proof (including that of \Cref{Prop:levelwise-wbc}) shows that the 2-category $\2Fun^{\mathsf{str}}_\kk (\cat B, \cat C)$ of \emph{strict} 2-functors, \emph{strict} transformations (\ie 2-natural transformations) and modifications is also Krull--Schmidt.
In fact, if the 2-category $\cat B$ is small (which we had better assume anyway!), it should be possible to show that the 2-full inclusion $\2Fun^{\mathsf{str}}_\kk (\cat B, \cat C) \subset \PsFun_\kk(\cat B, \cat C)$ is a biequivalence by adapting to the $\kk$-linear case the coherence results for 2-monads due to Power and Lack; see \cite{Power89} \cite{Lack02b} \cite{Lack07b}.
\end{Rem}

A \emph{$\kk$-linear tensor category} is a monoidal categories whose tensor product is a $\kk$-linear functor of both variables. For a fixed such tensor category~$\cat A$, we may consider the 2-category $\PsMod\cat A$ whose objects are \emph{\textup($\kk$-linear left\textup) module categories over~$\cat A$}, whose 1-morphisms are \emph{module functors} and 2-morphisms \emph{module natural transformations}; see \eg \cite{Greenough10} for detailed definitions.
(These notions specialize those of pseudomonoids and their left pseudomodules \cite{DayStreet97} \cite{McCrudden00} to the  symmetric monoidal 2-category $\CAT_\kk$ of $\kk$-linear categories, $\kk$-linear functors and natural transformations.)
In fact:

\begin{Lem} \label{Lem:exp}
For any $\kk$-linear monoidal category~$\cat A$, let $\mathrm B\cat A$ denote its delooping, \ie the bicategory with a single object having $\cat A$ as monoidal endo-category.
Then there is a canonical biequivalence (actually an isomorphism) of $\kk$-linear 2-categories between $\PsMod \cat A$ as defined above and $\PsFun_\kk (\mathrm B \cat A, \CAT_\kk)$ (\Cref{Ter:PsFun}).
\end{Lem}

\begin{proof}
The exponential law yields a bijection between objects of the 2-categories, with a (coherently associative and unital) pseudoaction $-\odot-\colon \cat A\otimes_\kk \cat M\to \cat M$ corresponding to a strong monoidal functor $\cat A\to \End_\kk(\cat M)$, that is to a pseudofunctor $\mathrm B\cat A\to \CAT_\kk$ sending the unique object of $\mathrm B\cat A$ to~$\cat M$.
It is immediate from the  definitions that module functors $F\colon \cat M_1\to \cat M_2$ and module natural transformations $\varphi\colon F\Rightarrow F'$ correspond to pseudonatural transformations and modifications.
\end{proof}

\begin{Rem} \label{Rem:PsMon-var}
By replacing $\CAT_\kk$ with a full 2-subcategory $\ADD_\kk$, $\ADDic_\kk$, $\mathsf{2FVect}_\kk$ etc.\ in the above, we can restrict attention to those tensor categories and their module categories which are additive, idempotent complete etc., as appropriate.
\end{Rem}

We conclude by specializing \Cref{Thm:KS-PsFun} to some examples of interest.

\begin{Cor} \label{Cor:tens-cat}
Let $\cat A$ be any finite (multi-) tensor category in the sense of Etingof--Ostrik \cite{Ostrik03} \cite{EtingofOstrik04}, over an algebraically closed field~$\kk$.
Then the 2-category of finite module categories over~$\cat A$ is Krull--Schmidt, as well as its 2-subcategory of exact module categories.
\end{Cor}

\begin{proof}
A finite tensor (or ``multi-tensor'') category $\cat A$ is in particular a $\kk$-linear tensor category in our sense above; it is one which is \emph{finite} as in \Cref{Exa:CATfin}.
Since the 2-subcategory $\CAT_\kk^\mathsf{fin}\subset \CAT_\kk$ of finite $\kk$-linear categories is Krull--Schmidt, so is $\PsFun_\kk(\mathrm B\cat A, \CAT_\kk^\mathsf{fin})$ by \Cref{Thm:KS-PsFun}. 
We conclude with \Cref{Lem:exp} and \Cref{Rem:PsMon-var} that finite $\cat A$-module categories form a Krull--Schmidt 2-category.

By definition, a module category~$\cat M$ is \emph{exact} if $P\odot M$ is a projective object in $\cat M$ for every $M\in \cat M$ and every projective $P\in \cat A$. Over a finite~$\cat A$, exact module categories are finite (\cite[Lemma~3.4]{Ostrik03}) and are clearly closed under taking direct summands, hence they form a full 2-subcategory which is also Krull--Schmidt.
\end{proof}

\begin{Cor} \label{Cor:ss}
Every finite semisimple 2-category in the sense of Douglas--Reutter \cite{DouglasReutter18pp} is Krull--Schmidt, and its indecomposable and simple objects coincide.
\end{Cor}

\begin{proof}
As in \cite{DouglasReutter18pp}, here we work over an algebraically closed field~$\kk$ of characteristic zero.
Then by \cite[Thms.\,1.4.8-9]{DouglasReutter18pp}, the 2-category of finite (left) semisimple module categories over a multifusion category~$\cat A$ is an example of a finite semisimple 2-category, and the latter are all of this form up to biequivalence. (The cited result is formulated for right module categories, but it holds equally well for left ones because multifusion categories are stable under reversing their tensor product).
By definition, a multifusion category is just a finite semisimple tensor category, thus in particular a $\kk$-linear tensor category~$\cat A$ as above.
By \Cref{Lem:exp}, with $\CAT_\kk$ replaced by $\mathsf{2FVect}_\kk$ (\Cref{Rem:PsMon-var}), we see that the  2-category of finite semisimple module categories over $\cat A$ is biequivalent to $\PsFun_\kk (\mathrm B \cat A, \mathsf{2FVect}_\kk)$. 
Since $\mathsf{2FVect}_\kk$ is Krull--Schmidt (\Cref{Exa:2vect}), we can therefore conclude with \Cref{Thm:KS-PsFun}.

It is easily checked that simples and indecomposables agree.
\end{proof}

\begin{Rem}
We found it satisfying to derive \Cref{Cor:ss} from \Cref{Thm:KS-PsFun}, but this is overkill: More directly, and more generally, it follows from the results in \cite{DouglasReutter18pp} that all \emph{locally finite} semisimple 2-categories are also Krull--Schmidt.
\end{Rem}

Recall now that a \emph{2-groupoid} is a bicategory in which every 1-cell is an equivalence and every 2-cell an isomorphism. It is a \emph{2-group} if it has a single object, \ie if it is the delooping of a monoidal groupoid whose objects are tensor-invertible.

\begin{Cor} \label{Cor:2reps}
Let $\kk$ be any field.
Let $\cat G$ be a 2-groupoid with only finitely many equivalence classes of objects (\eg a 2-group).
Then the bicategory 
\[ \mathsf{2FRep}_\kk \mathcal (\cat G) := \PsFun(\cat G, \mathsf{2FVect}_\kk) 
\]
of finite dimensional $\kk$-linear 2-representations of $\cat G$ is Krull--Schmidt.
\end{Cor}

\begin{proof}
Consider the (easily constructed) free $\kk$-linearization of~$\cat G$, which is a $\kk$-linear bicategory $\kk \cat G$ with the same objects as $\cat G$ and equipped with a pseudofunctor $\cat G\to \kk \cat G$ inducing a biequivalence
$
\PsFun_\kk (\kk \cat G , \cat C) \overset{\sim}{\to} \PsFun ( \cat G , \cat C)
$ 
for every $\kk$-linear bicategory~$\cat C$.
In particular, for $\cat C = \mathsf{2FVect}_\kk$ we obtain a biequivalence between  $\mathsf{2FRep}_\kk \mathcal (\cat G)$ and a Krull--Schmidt bicategory as in \Cref{Thm:KS-PsFun}.
\end{proof}

\begin{Rem}
By varying the 2-groupoid~$\cat G$, we can specialize $\mathsf{2FRep}_\kk(\cat G)$ to an impressive variety of families of interesting 2-categories already over an algebraically closed field of characteristic zero, \ie when we are at the semisimple intersection of \Cref{Cor:2reps} and \Cref{Cor:ss}. 
See the detailed bestiary in \cite[\S1.4.5]{DouglasReutter18pp}.
\end{Rem}

\begin{Rem}
By replacing $\mathsf{2FVect}_\kk$ with $\mathsf{2FHilb}$ (\Cref{Exa:2hilb}) we obtain the Krull--Schmidt property for \emph{unitary} finite dimensional representations of 2-groups. 
\end{Rem}

\begin{Cor} \label{Cor:MM}
Let $\cat B$ be any $\kk$-finitary 2-category in the sense of Mazorchuk--Miemietz \cite{MazorchukMiemietz11}, over an algebraically closed field~$\kk$.
Then the 2-category $\cat B \mathfrak{-mod}$ of $\kk$-linear 2-representations of $\cat B$ is Krull--Schmidt.
\end{Cor}

\begin{proof}
By definition $\cat B \mathfrak{-mod} := \2Fun^\mathsf{str}_\kk (\cat B, \CAT^{\mathsf{fin}}_\kk)$ as in \Cref{Rem:KS-2Fun} and $\cat B$ has finitely many objects, hence we may once again apply \Cref{Thm:KS-PsFun}.
\end{proof}

Many further variations are possible, for instance one may wish to study finite ``derived'' rather than abelian or semisimple 2-representations by using the Krull--Schmidt 2-category $\mathsf{TRI}^{\mathsf{fin}}_\kk$ as target (\Cref{Exa:CATtrfin}).
As sampled above, most directions of 2-representation theory found in the literature impose strong finiteness conditions (often motivated by combinatorial goals) which ensure the Krull--Schmidt property. 
We trust interested readers to adapt our proofs to any such situations.

%------------------------------------------------------------------------------

\bibliographystyle{alpha}
%\bibliography{articles}

%------------------------------------------------------------------------------
\end{document}